\documentclass[11pt]{article}

\usepackage[T1]{fontenc} 
\usepackage[english]{babel} 
\usepackage{microtype} 
\usepackage[a4paper, left=2.8cm, right=2.8cm, top=2.8cm, bottom=2.8cm]{geometry} 
\usepackage{amsmath,amssymb,amsthm}
\usepackage{mathrsfs}
\usepackage{mathtools}
\usepackage{mhequ} 

\usepackage{hyperref} 
\hypersetup{hidelinks}

\theoremstyle{plain}
\newtheorem{theorem}{Theorem}[section]

\newtheorem{lemma}[theorem]{Lemma}
\newtheorem{proposition}[theorem]{Proposition}

\theoremstyle{definition}
\newtheorem{definition}[theorem]{Definition}

\theoremstyle{remark}
\newtheorem{remark}[theorem]{Remark}

\numberwithin{equation}{section}

\newcommand*{\eqdef}{\stackrel{\text{\scriptsize{def}}}{=}}
\newcommand{\dd}{\mathrm{d}}

\author{
  \bf Damiano De Gaspari\\
  \texttt{ \small damiano.de.gaspari@tuwien.ac.at}\\
  \small{Technical University of Vienna, Austria}
  \and
  \bf Levi Haunschmid-Sibitz\\
  \texttt{\small levi.haunschmid@tuwien.ac.at}\\
  \small{Technical University of Vienna, Austria}
}

\title{$(\log t)^\frac{2}{3}$ - superdiffusivity for the 2d stochastic Burgers equation}
\date{November 29, 2024}

\begin{document}

\maketitle

\begin{abstract}
The Stochastic Burgers equation was introduced in [H. van Beijeren, R. Kutner and H. Spohn, Excess noise for driven diffusive systems, PRL, 1985] as a continuous approximation of the fluctuations of the asymmetric simple exclusion process. It is formally given by 
\[\partial_t\eta =\frac{1}{2}\Delta\eta+ \mathfrak w\cdot\nabla(\eta^2) + \nabla\cdot\xi \, ,\]
where $\xi$ is $d$-dimensional space time white noise and $\mathfrak w$ is a fixed non-zero vector. In the critical dimension $d=2$ at stationarity, we show that this system exhibits superdiffusive behaviour: more specifically, its bulk diffusion coefficient behaves like $(\log t)^\frac23$, in a Tauberian sense, up to $\log\log\log t$ corrections. This confirms a prediction made in the physics literature and complements [G. Cannizzarro, M. Gubinelli, F. Toninelli, Gaussian Fluctuations for the stochastic Burgers equation in dimension $d\geq 2$, CMP, 2024], where the same equation was studied in the weak-coupling regime. Furthermore this model can be seen as a continuous analogue to [H.T. Yau, $(\log t)^\frac{2}{3}$ law of the two dimensional asymmetric simple exclusion process, Annals of Mathematics, 2004].
\end{abstract}


\tableofcontents

\section{Introduction}
We study the stochastic Burgers equation formally given by
\begin{equ}\label{eq:formal}
	\partial_t\eta =\frac{1}{2}\Delta\eta+ \mathfrak w\cdot\nabla(\eta^2)  + \nabla\cdot\xi \, ,
\end{equ}
where $\eta=\eta(t,x)$ is a scalar field depending on time $t$ and space $x\in\mathbb R^d$, with $d\geq 1$, $\mathfrak w\in\mathbb R^d$ is a fixed vector controlling the strength and direction of the nonlinearity and $\xi=(\xi_1,\dots,\xi_d)$ is $d$-dimensional space-time white noise, namely the centered Gaussian process whose covariance structure is formally given by 
\begin{equ}
    \mathsf E(\xi_i(t,x)\xi_j(s,y)) = 
    \begin{cases}
        \delta(t-s)\delta(x-y),\quad & \text{if $i=j$}\\
        0, &\text{else,}
    \end{cases}
\end{equ}
for $i,j \in \{1,\dots,d\}$, $t,s\in [0,\infty)$, $x,y\in \mathbb R^d$ and where $\delta$ is the Dirac delta function.
This equation was introduced in \cite{van1985excess} as a proposed  continuum analogue of the fluctuations of driven diffusive systems with one conserved quantity, like the Asymmetric Simple Exclusion Process. In this work we will always consider~(\ref{eq:formal}) for $\eta_0\sim \underline \eta$, the law of $\mathbb R^d$ space white noise. For this choice of initial distribution, the solution is formally stationary and distributed according to $\underline \eta$ at all times. 

In dimension $d=1$ this equation is equivalent to the space derivative of the Kardar-Parisi-Zhang (KPZ) equation, for which there has been a tremendous amount of new results in the last decade.
In particular there is a global pathwise solution theory, see \cite{HairerKPZ, GubinelliPerkowskiKPZ}, and much more is known, see e.g. the surveys \cite{QuastelSpohn15KPZSurvey,CorwinKPZsurvey} and the references therein.
In particular the connection with the discrete models (in particular the Weakly Asymmetric Exclusion Process on $\mathbb{Z}$) is well established.
See \cite{bertiniGiacomin1997KPZandparticles,hairerQuastel2018growthKPZ,gonccalvesPerkowski2020derivation} for works connecting particle and growth models to one-dimensional KPZ.
For a study of the bulk diffusivity and similar quantities for $d=1$ see \cite{BQS2011KPZfluctuations}.
In dimension $d\geq 3$ the recent work \cite{cannizzaro2023gaussian} establishes Gaussian fluctuations at large scales. The analogous result for asymmetric simple exclusion processes was proven before in \cite{EMY1994diffusive,LY1997fluctuation,CLO2001equilibrium}.

Dimension $d=2$ is of particular interest for several reasons. First of all, for $d\geq 2$ it falls outside the domain of applicability of both the method of regularity structures developed in \cite{hairer2014theory} and the paracontrolled distribution method of \cite{gubinelliperkowski2015paracontrolled,gubinelliperkowski2018introduction}.
Moreover, it is the \emph{critical} dimension in the sense of scaling, as we will further discuss in Section~\ref{se:scaling}.
It is also the model which should describe the fluctuations of 2d ASEP, for which $(\log t)^{\frac23}$ superdiffusivity was shown in \cite{Yau2004logtProcess}.
The recent work \cite{cannizzaro2023gaussian} studies the weak coupling regime of the 2d stochastic Burgers equation, i.e. the size of the nonlinearity is scaled down while looking at larger and larger scales.
In this regime they also find non-trivial Gaussian fluctuations, in the sense that the limiting equation is a stochastic heat equation with modified Laplacian, that depends on the nonlinearity.
This result suggests superdiffusivity for the strong coupling case, i.e. when the nonlinearity is not scaled down, but it does not imply it.
This is the case to which the present case is devoted.
Our result can be seen as an analog of \cite{Yau2004logtProcess} in the continuum and is also the first critical SPDE for which $(\log t)^\frac23$ superdiffusivity has been proven, to the best of the authors' knowledge.
While our estimates remain technical, we manage to avoid the splitting of sums into various good and bad regions, which has been a major obstacle to replicating the success of \cite{Yau2004logtProcess} to other models.
Also, compared to \cite{Yau2004logtProcess}, the sub-leading corrections to the $(\log t)^\frac23$ behavior are of lower order.

In general, bulk diffusion coefficients have been conjectured to diverge either like $(\log t)^\frac12$ or like $(\log t)^\frac23$ for a wide variety of models in the critical dimension, see e.g. \cite{toth2012superdiffusive,LRY2005superdiffusivity,wainwright1971decay}.
Recent successes in proving $(\log t)^\frac12$ superdiffusivity are \cite{ABK2024superdiffusiveCLT, CET2023stationary, CHSTDiffusionCurl, LimaWeber2024brownian}.
The distinction between these two classes is characterized by their symmetries.
The models in the $(\log t)^\frac23$ universality class have one direction in which the system behaves superdiffusively, while in the orthogonal direction it behaves diffusively.
In our case this direction is given by the vector $\mathfrak w$.
The models in the $(\log t)^\frac12$  class, on the other hand, often have some kind of rotational symmetry and behave superdiffusively in every direction.

Finally let us summarize the structure of this paper.
In the following Subsection~\ref{se:scaling} we rigorously define the equation and the bulk diffusion coefficient and state the main theorem.
Then, in Section~\ref{se:prelim}, we set up notation and recall elements of Gaussian analysis and the form of the generator.
In Section~\ref{se:truncres} we reduce the problem to estimating certain operators on Fock space.
Then, in Section~\ref{se:estimates}, we prove iterative estimates of these operators, which we use in Section~\ref{se:mainproof} to prove the main theorem. Finally, Appendix~\ref{ap:replacement_lemmas} gathers some self-contained technical results useful for Section~\ref{se:estimates}, while Appendix~\ref{ap:bulkheuristic} presents a heuristic explanation that motivates the expression for the bulk diffusivity given by~(\ref{eq:bulkD}) below.

\subsection{Scaling, Regularization and Green-Kubo formula}\label{se:scaling}

As it is written, equation (\ref{eq:formal}) is ill-posed, since any solution would be too irregular for the nonlinearity to be well-defined.
Since we are interested in the large scale behaviour, we regularize the nonlinearity at small scales and then consider larger and larger scales.
We do so by introducing Fourier cut-offs inside and outside the nonlinearity:
\begin{equ}\label{eq:fullspace}
	\partial_t\eta=\frac12\Delta\eta+\mathfrak w \cdot \Pi_1\nabla(\Pi_1\eta)^2+\nabla\cdot\xi\,,
\end{equ}
where for $a>0$, the Fourier cut-off $\Pi_a$ acts on $\eta$ in Fourier by cutting modes larger than $a$, i.e.
\begin{equ}
	\widehat{\Pi_a\eta}(k)\eqdef\widehat{\eta}(k)\mathbf{1}_{|k|\leq a}\,.
\end{equ}
Additionally, in order to avoid integrability issues arising in infinite volume, we study the equation on a large torus $\mathbb T^2_N$ of side-length $2\pi N$.
We will later let $N$ go to infinity, see Theorem~\ref{th:main}. We recall that we are always considering the equation at stationarity, namely with initial condition distributed according to the law of space white noise $\underline \eta$. Indeed, regularizing the nonlinearity as described above does not change the fact that $\underline \eta$ is invariant under the dynamics, formally for the equation on full space~(\ref{eq:fullspace}) and rigorously for the one on the torus~(\ref{eq:toroN}), see Lemma~\ref{le:generator} below.    
For equation~(\ref{eq:fullspace}) we define the bulk diffusivity using a Green-Kubo formula justified in Appendix~\ref{ap:bulkheuristic}:
\begin{equ}\label{eq:bulkD}
	D^N(t)\eqdef
	1+\frac{\lvert\mathfrak w\rvert^2}{t}\int_0^{t}\int_0^s\int_{\mathbb T^2_N}\mathbf{E}\left(\Pi_1{:}(\Pi_1\eta)^2{:}(r,x)\Pi_1{:}(\Pi_N\eta)^2{:}(0,0) \right)\dd x \dd r \dd s\,,
\end{equ}
where $\mathbf{E}$ denotes the expectation with respect to the stationary solution started from mean-zero white noise, and ${:}X^2{:}$ denotes the Wick product, which in this case just subtracts the expectation, i.e. ${:}X^2{:}=X^2-\mathbf{E}(X^2)$.
Heuristically, the bulk diffusivity coefficient measures how correlations spread out in space as a function in time.

For convenience, we work on the torus with side-length $2\pi$.
To do so, define the rescaled solution $\eta^N:\mathbb R_+\times\mathbb T^2_1\to\mathbb R$ by
\begin{equ}\label{eq:scaling}
	\eta^N(t,x)=N\eta\left(N^2t,Nx\right)\,,
\end{equ}
which solves the equation
\begin{equ}[eq:toroN]
	\partial_t\eta^N =\frac{1}{2}\Delta\eta^N+ \mathfrak w\cdot\Pi_N\nabla\left(\Pi_N\eta^N\right)^2  + \nabla\cdot\xi \, .
\end{equ}
Expressing the bulk diffusivity from (\ref{eq:bulkD}) in terms of $\eta^N$ leads, after a suitable change of variables, to  the expression
\begin{equ}\label{eq:bulkDN}
	D^N(t)=
	1+N^2\frac{\lvert\mathfrak w\rvert^2}{t}\int_0^{\frac{t}{N^2}}\int_0^s\int_{\mathbb T^2_1}\mathbf{E}\left(\Pi_N{:}(\Pi_N\eta^N)^2{:}(r,x)\Pi_N{:}(\Pi_N\eta^N)^2{:}(0,0) \right)\dd x \dd r \dd s\,.
\end{equ}
Our main theorem concerns the Laplace transform of $D^N$, defined by
\begin{equ}\label{eq:LapD}
	\mathcal{D}^N(\lambda)\eqdef\int_0^\infty e^{-\lambda t}tD^N(t)\dd t\,.
\end{equ}
Note that this is the standard Laplace transform instead of the one used in \cite{CET2023stationary}, but the two definitions only differ by a factor of $\lambda$.

\begin{theorem}\label{th:main}
	Let $\mathfrak w\neq 0$ and let $\eta^N$ be the stationary solution to \eqref{eq:toroN} as considered above. Define the Laplace transform $\mathcal D^N$ of the bulk diffusivity as in (\ref{eq:LapD}).
	Then, for every $\delta\in(0,1)$, there is a constant $C = C(\lvert \mathfrak w\rvert)$ such that, for all $\lambda$ small enough,
	\begin{equ}
		\limsup_{N\to\infty}\mathcal{D}^N(\lambda)\leq \frac{C}{\lambda^2}\left(\log\log\lvert\log\lambda\rvert\right)^{3+\delta}\lvert\log\lambda\rvert^\frac23
	\end{equ}
	and
	\begin{equ}
		\liminf_{N\to\infty}\mathcal{D}^N(\lambda)\geq \frac{1}{C\lambda^2}\left(\log\log\lvert\log\lambda\rvert\right)^{-3-\delta}\lvert\log\lambda\rvert^\frac23\,.
	\end{equ}
\end{theorem}
Note that by translating \cite[Lemma 1]{quastelValko2008note} into our setting, the upper bound gives $D^N(t)\lesssim (1+\log(1+t))^{\frac23+o(1)}$ as $t\uparrow\infty$.
For the lower bound such a statement is not true in general.
Note however that $\mathcal D^N(\lambda)\sim\frac{C}{\lambda^2}\lvert\log\lambda\rvert^\frac23$ as $\lambda \downarrow 0$ would imply $\frac{2C}{T}\int_0^T tD^N(t)\dd t\sim T(\log{T})^\frac23$ as $T \uparrow \infty$ by general Tauberian inversion theorems, see \cite[Chapter XIII.5]{feller1991introduction}.
Thus, the theorem says that $D(t)$ grows like $(\log t)^\frac23$, at least in a weak Tauberian sense.

Note also that the correction terms $(\log\log\lvert\log\lambda\rvert)^{\pm 3\pm \delta}$ of Theorem~\ref{th:main} are of lower order with respect to the ones of the corresponding result in \cite{Yau2004logtProcess}, which are $e^{\pm \gamma(\log\log\lvert \log \lambda \rvert)^2}$ for some constant $\gamma>0$.

\subsection{Sketch of the main proof}
The basic structure of the proof follows the method first used in \cite{LQSY2004Superdiffusivity,Yau2004logtProcess,CET2023stationary,CHSTDiffusionCurl} and since then successfully applied in \cite{CET2023stationary,CHSTDiffusionCurl}.
It consists of the following steps.
First, we express $\mathcal D^N(\lambda)$ as an inner product in the $L^2$ space associated to the stationary measure $\underline\eta$, see Proposition~\ref{pr:Laplace_bulk} below.
This inner product has the form $\langle\phi,(\lambda-\mathcal L)^{-1}\phi\rangle$, for a specific $\phi$.
We then estimate it, using the Chaos decomposition and a Lemma due to \cite{LQSY2004Superdiffusivity}, restated as \ref{le:yau} below, which states that truncating the generator in chaos gives upper and lower bounds for $\mathcal D^N(\lambda)$.
This leads to an iterative estimation scheme, which results in Theorem~\ref{th:iterative_bounds}, which we then use to prove Theorem~\ref{th:main}.
This iterative estimation scheme is inspired by the methods of \cite{LQSY2004Superdiffusivity,Yau2004logtProcess,CET2023stationary,CHSTDiffusionCurl}.
However, the expressions of our upper and lower bounds are different, see Theorem~\ref{th:iterative_bounds} and the definitions in that section.
In particular, compared to \cite{CET2023stationary,CHSTDiffusionCurl} we do not absorb the off-diagonal terms into the main term, but instead estimate them separately.

The above procedure relies heavily on Gaussian Analysis, which has been used successfully to understand a variety of critical and super-critical SPDEs and related models via their generator, see \cite{jinperkowski2024fractional,cannizzaroGiles2024invariance,gubinelli2020infinitesimal}.

\section{Preliminaries}\label{se:prelim}

\subsection{Notation} \label{su:notation}
Recall that for $N>0$ we denote by $\mathbb T^2_N$ the torus of side-length $2\pi N$.
If $N=1$ we write $\mathbb T^2$ instead of $\mathbb T^2_1$.
Let $(e_k)_{k\in\mathbb Z^2}$ be the standard Fourier basis on $\mathbb T^2$, i.e. $e_k(x)=\frac{1}{2\pi}\exp(ik\cdot x)$, which constitute an orthonormal basis of $L^2(\mathbb T^2)$.
The Fourier transform of a function $\varphi\in L^2(\mathbb T^2)$, denoted by $\widehat{\varphi}$, is given by
\begin{equ}
    \widehat\varphi(k)\eqdef \int_{\mathbb T^2}\varphi(x)e_{-k}(x)\dd x \qquad \text{for $k\in\mathbb Z^2$}\,.
\end{equ}

Moreover, we denote by $k_{1:n}$ the sequence $(k_1,\dots,k_n)$, where $k_i\in\mathbb{Z}^2$.
For example, for an $L^2$ function $f$ on $(\mathbb{T}^2)^n$, we write its Fourier transform as $\widehat f(k_{1:n})=\widehat{f}(k_1,\dots,k_n)$.
Furthermore we define $|k_{1:n}|^2\eqdef\sum_{i=1}^n|k_i|^2$.

We denote by $\mathbb P$ and $\mathbb E$ the law and the corresponding expectation of the stationary measure given by mean-zero spatial white noise, as will be defined in subsection~\ref{se:fockspace}.
With $\mathbf{P}$ and $\mathbf{E}$ we denote instead the law and the corresponding expectation of the process given by the solution of (\ref{eq:toroN}) started from the aforementioned stationary measure.

Finally, given $A,B\in\mathbb R$, we write $A\lesssim B$ if there exists an absolute constant $c>0$, independent of all variables on which $A$ and $B$ may depend, such that $A \le c B$. In particular, we will only use this notation if $c$ is independent of $\mathfrak w$.

\subsection{Chaos Decomposition}\label{se:fockspace}

Let $(\Omega,\mathscr{F},\mathbb P)$ be a complete probability space and $\eta$ be real-valued mean-zero spatial white noise on $\mathbb T^2$, i.e. $\eta$ is the Gaussian field with covariance
\begin{equ}\label{eq:realcov}
	\mathbb{E}\left(\eta(\varphi)\eta(\psi)\right)=\langle\varphi,\psi\rangle_{L^{2}(\mathbb{T}^2)}\,,
\end{equ}
where $\varphi$ and $\psi$ belong to $L^2_0(\mathbb T^2)$, the space of square-integrable real-valued functions that integrate to $0$.
Since we work in Fourier, we also want to test $\eta$ against complex valued functions. We do this by setting $\eta(\varphi)=\eta(\mathrm{Re}(\varphi))+\iota \eta(\mathrm{Im}(\varphi))$ every for $\varphi\in L^2(\mathbb T^2;\mathbb C)$, where $\iota = \sqrt{-1}$ denotes the imaginary unit. This leads to considering the covariance function (which extends (\ref{eq:realcov}))
\begin{equ}
\mathbb{E}\left(\eta(\varphi)\overline{\eta(\psi)}\right)=\langle\varphi,\psi\rangle_{L^{2}(\mathbb{T}^2;\mathbb{C})}\,,
\end{equ}
where the inner product is the standard sesquilinear inner product of square-integrable complex valued functions (and $\varphi$ and $\psi$ still integrate to $0$).
Note that $\eta$ is still real-valued in the sense that $\overline{\eta(\varphi)}=\eta(\overline{\varphi})$, which would not be the case for complex-valued white noise, see e.g. \cite[Section 1.4]{Janson1997GaussianSpaces}.
Using this extension, we define $\widehat{\eta}(k)=\eta(e_{-k})$.
These are complex valued Gaussian variables satisfying $\overline{\widehat{\eta}(k)}=\widehat{\eta}(-k)$ and $\mathbb{E}\left(\widehat{\eta}(j)\overline{\widehat{\eta}(k)}\right)=\delta_{j,k}$.
Since we only test against mean-zero functions, $\widehat{\eta}(0)$ is not defined and we set it to $0$.

Let $L^2(\eta)$ be the space of $L^2$ random variables on $\Omega$ measurable with respect to the $\sigma$-algebra generated by $\eta$.
For $n\in\mathbb N$, let $\mathscr{H}_n$ be the \emph{$n$-th homogeneous Wiener chaos}, i.e. the closed linear subspace of $L^2(\eta)$ generated by the random variables $H_n(\eta(h))$, where $H_n$ is the $n$-th Hermite polynomial and $h$ is a mean-zero test function of norm $1$.
By \cite[Theorem 1.1.1]{Nualart}, $L^2(\eta)=\bigoplus_{n\geq0}\mathscr{H}_n$ is an orthogonal Hilbert space decomposition of $L^2(\eta)$.
Define also $\Gamma L^2=\bigoplus_{n\geq0}\Gamma L^2_n$, where $\Gamma L^2_n$ is the $n$-fold symmetric tensor product of $L^2_0(\mathbb{T}^2)$, i.e. the space of symmetric $L^2$ functions $f$ on $(\mathbb T^2)^n$ which are mean-zero in every variable, i.e. such that $\int_{\mathbb T^2}f(x,y_{1:n-1})\dd x=0$ for every $y_{1:n-1}\in(\mathbb T^2)^{n-1}$.
By \cite[Proposition 1.1.1]{Nualart}, there is a canonical isometry $I$ between $\Gamma L^2$ and $L^2(\eta)$, whose restrictions $I_n$ to $\Gamma L^2_n$ are isometries between $\Gamma L^2_n$ and $\mathscr{H}_n$.
This gives the following correspondence: for every $F\in L^2(\eta)$ there is a family of kernels $(f_n)_{n\geq 0}\in\Gamma L^2$ such that $F=\sum_{n\geq 0}I_n(f_n)$ and
\begin{equ}
	\mathbb{E}\left(F^2\right)=\|(f_n)_{n\geq 0}\|_{\Gamma L^2}^2\eqdef\sum_{n\geq 0}n!\|f_n\|^2_{L^2\left((\mathbb{T}^2)^n\right)}\,.
\end{equ}
Here the right-hand side also defines the $\Gamma L^2$ inner product.

\begin{remark}
    As opposed to the introduction, where $\eta$ was used as the formal solution, and $\underline{\eta}$ as the stationary law, we switch to using $\eta^N$ for the solution of \eqref{eq:toroN} and $\eta$ as the stationary law, to reduce visual clutter.
\end{remark}

\begin{remark}
	By this isometry between $L^2(\eta)$ and the Fock space $\Gamma L^2$, we will identify throughout the paper operators acting on either space by composing them with $I$ or $I^{-1}$ as appropriate (and without mentioning that we are doing so).
\end{remark}

\begin{remark}
	It is not strictly necessary to take the white noise to be mean-zero, but it is natural since the dynamics of the system are conservative.
	If we start equation (\ref{eq:toroN}) from a standard white noise $\eta_0$ (i.e. $\widehat{\eta}(0)$ is a standard Gaussian), then, for any future time $t$, we have $\widehat{\eta}_t(0)=\widehat{\eta}_0(0)$. 
        Moreover, $\widehat{\eta}_t(0)$ is independent of all other $\widehat{\eta}_t(k)$.
	Therefore we can just set it to $0$.
	In terms of Fourier kernels it means that for any $\varphi\in\mathscr H_n$ it holds that $\widehat{\varphi}(k_{1:n})$ is $0$ if any of the $k_1,\dots,k_n$ are $0$.
\end{remark}

\subsection{The Generator}

In this subsection we recall the generator and some properties of \eqref{eq:toroN} from \cite{cannizzaro2023gaussian}.
The following is (part of) Lemma 2.1 and Lemma 2.2. from \cite{cannizzaro2023gaussian}.

\begin{lemma}\label{le:generator}
    For every deterministic initial condition $\eta_0$, the solution $t\mapsto \eta^N_t$ of (\ref{eq:toroN}) exists globally in time and is a strong Markov process.
    Its generator can be written as $\mathcal{L}^N=\mathcal{L}_0+\mathcal{A}_+^N+\mathcal{A}_-^N$, where $\mathcal L_0$ is symmetric with respect to $\mathbb{P}$, $\mathcal (A_+^N)^*=-\mathcal A_-^N$, again with respect to $\mathbb P$, and the operators $\mathcal{L}_0$, $\mathcal{A}_+^N$ and $\mathcal{A}_-^N$ act on $\varphi\in\mathscr{H}_n$ as:
	\begin{equs}\label{eq:gen_Fourier}
		\widehat{\mathcal L_0 \varphi} (k_{1:n})&=-\tfrac12\lvert k_{1:n}\rvert ^2\widehat \varphi(k_{1:n}) \\
		\widehat{\mathcal A_+^N \varphi} (k_{1:n+1}) &= - \frac{\iota}{\pi(n+1)} \! \! \sum_{1 \le i < j \le n+1} \! \! \! \! \! \! \mathbb J^N_{k_i,k_j} \left[\mathfrak w \cdot (k_i + k_j)\right] \widehat \varphi \left( k_i + k_j, k_{\{1:n+1\}\setminus \{i,j\}} \right) \\
		\widehat{\mathcal A^N_- \varphi} (k_{1:n-1})&=-\frac{\iota \, n}{\pi}\,\sum_{j=1}^{n-1}(\mathfrak w\cdot k_j) \! \sum_{\ell+m=k_j} \!\mathbb J^N_{\ell,m} \, \widehat \varphi\left(\ell,m,k_{\{1:n-1\}\setminus\{j\}}\right)\,,
	\end{equs}
	where the indicator function $\mathbb J$ is given by
	\begin{equ}[eq:J]
		\mathbb J_{\ell, m} \eqdef \mathbf 1_{\{0< \lvert \ell \rvert\le N, 0<\lvert m\rvert\le N,0<\lvert \ell+m\rvert\le N\}}\,.
	\end{equ}
    Additionally, the mean-zero white noise $\eta$ defined by (\ref{eq:realcov}) satisfies $\eta \mathcal L = \eta$ and thus it is a stationary law of the Markov process whose evolution is governed by the generator $\mathcal L$ and whose initial condition is distributed according to $\eta$.
\end{lemma}
Note that the expression for the Fourier multiplier of $\mathcal{A}_-$ given above differs from the one of \cite{cannizzaro2023gaussian} by a minus sign.
This is due to a typo in \cite{cannizzaro2023gaussian}.
\begin{remark}
    In fact, the process of Lemma~\ref{le:generator} assumes values in some Besov space of negative regularity, see \cite[Theorem 4.5]{CESnontriviality} where this is proved for the AKPZ equation.
\end{remark}

\section{Truncated Resolvent Equation}\label{se:truncres}
In this section we first rewrite $\mathcal D^N(\lambda)$, defined in \eqref{eq:LapD}, in terms of a resolvent.
Then we reduce the study of this quantity to estimating certain operators on Fock space.

The following proposition allows to express $\mathcal{D}^N$ just in terms of the stationary measure and the generator.
\begin{proposition} \label{pr:Laplace_bulk}
        Let  $\eta^N$ be the stationary solution to (\ref{eq:toroN}) started from mean-zero white noise.
	The Laplace transform of the bulk diffusivity is given by
	\begin{equ}
		\mathcal{D}^N(\lambda)=\frac{1}{\lambda^2}+\frac{1}{\lambda^2}\lvert\mathfrak w\rvert^2\mathbb{E}\left( \tilde{\mathcal{N}}^N[\eta]\left(\lambda N^2-\mathcal L^N\right)^{-1}\tilde{\mathcal{N}}^N[\eta]\right)\,,
	\end{equ}
	where $\tilde{\mathcal{N}}^N[\eta]\in\mathscr{H}_2$ is purely in the second chaos and given by
	\begin{equ} \label{eq:n}
		\tilde{\mathcal{N}}^N[\eta]\eqdef\frac{1}{2\pi}\int_{\mathbb T^2}\Pi_N{:}(\Pi_N \eta)^2{:}(x)\dd x = \sum_{\substack{\ell+ m=0\\0<\lvert\ell\rvert\leq N}} {:}\widehat\eta(\ell)\widehat\eta(m){:}
	\end{equ}
	and its kernel $\mathfrak n^N=I^{-1}\left(\tilde{\mathcal{N}}^N[\eta]\right)\in \Gamma L^2_2$ is given in Fourier by
	\begin{equ} \label{eq:n_Fourier}
		\widehat{\mathfrak n^N}(j_1,j_2) = \mathbf{1}_{\{0<\lvert j_1 \rvert \le N, \, j_1+j_2 = 0\}} \, .
	\end{equ}
\end{proposition}
\begin{proof}
	Multiplying (\ref{eq:bulkDN}) by $t$ yields
	\begin{equ}\label{eq:bulkDt}
		tD^N(t)=
		t+N^2\lvert\mathfrak w\rvert^2\int_0^{\frac{t}{N^2}}\int_0^s\int_{\mathbb T^2}\mathbf{E}\left(\Pi_N{:}(\Pi_N\eta^N)^2{:}(r,x)\Pi_N{:}(\Pi_N\eta^N)^2{:}(0,0) \right)\dd x \dd r \dd s\,.
	\end{equ}
	Since the stationary process $\eta^N$ is translation invariant in space, we can write the spatial integral in the expression above as
	\begin{equs}
		\frac{1}{4\pi^2}\int_{\mathbb T^2}\int_{\mathbb T^2}&\mathbf{E}\left(\Pi_N{:}(\Pi_N\eta^N)^2{:}(r,x+y)\Pi_N{:}(\Pi_N\eta^N)^2{:}(0,y)\right) \dd x \dd y\\
		&=\mathbf{E}\left(\left(\frac{1}{2\pi}\int_{\mathbb T^2}\Pi_N{:}(\Pi_N\eta^N)^2{:}(r,x)\dd x\right)\left(\frac{1}{2\pi}\int_{\mathbb T^2}\Pi_N{:}(\Pi_N\eta^N)^2{:}(0,x)\dd x\right)\right)\\
            &=\mathbf{E}\left(\tilde{\mathcal{N}}^N[\eta^N(r)]\tilde{\mathcal{N}}^N[\eta^N(0)]\right)\,.
	\end{equs}
	Using the stationarity of $\eta^N$ we note
	\begin{equ}
		\int_0^{\frac{t}{N^2}}\int_0^s\mathbf{E}\left(\tilde{\mathcal{N}}^N[\eta^N(r)]\tilde{\mathcal{N}}^N[\eta^N(0)]\right)\dd r\dd s=\frac12\mathbf{E}\left(\left(\int_0^\frac{t}{N^2}\tilde{\mathcal{N}}^N[\eta^N(s)]\dd s\right)^2\right) \, .
	\end{equ}
	Using this to rewrite (\ref{eq:bulkDt}) and applying the Laplace transform gives:
	\begin{equs}
		\mathcal D^N(\lambda)
		&=\int_0^\infty e^{-\lambda t}tD^N(t)\dd t\\
		&=\int_0^\infty e^{-\lambda t}\left(t+\frac{1}{2}N^2\lvert\mathfrak w\rvert^2\mathbf{E}\left(\int_0^\frac{t}{N^2}\tilde{\mathcal{N}}^N[\eta^N(s)]\dd s\right)^2\right)\dd t\\
		&=\frac{1}{\lambda^2}+\lvert\mathfrak w\rvert^2\frac{1}{\lambda^2}\mathbb{E}\left(\tilde{\mathcal{N}}^N(\lambda N^2-\mathcal L^N)^{-1}\tilde{\mathcal{N}}^N\right)\,,
	\end{equs}
	where in the final step we used \cite[Lemma 5.1]{CESnontriviality}, which allows us to go from an expectation with respect to the process to one just with respect to the stationary measure.
\end{proof}

Since $I(\tilde{\mathcal{N}}^N)=\mathfrak{n}^N$ by Proposition~\ref{pr:Laplace_bulk} and $I$ is an isometry of $L^2(\eta)$ and $\Gamma L^2$, the expectation above can be expressed as 
\[
\mathbb{E}\left(\tilde{\mathcal{N}}^N(\lambda N^2-\mathcal L^N)^{-1}\tilde{\mathcal{N}}^N\right)=\langle \mathfrak{n}^N,(\lambda N^2-\mathcal L^N)^{-1}\mathfrak{n}^N\rangle\,.
\]
The presence of the operators $\mathcal{A}^N_+$ and $\mathcal A_-^N$ in $\mathcal L^N$ means finding $(\lambda N^2-\mathcal L^N)^{-1}\mathfrak{n}^N$ involves all chaoses, even though $\mathfrak{n}^N$ is purely in $\mathscr H_2$.
This makes finding an explicit expression for this term very difficult.
To overcome this we will apply a technique first used in \cite{LQSY2004Superdiffusivity},
which consists in truncating the generator in chaos, and then using iterative estimates to obtain upper and lower bounds on the truncated resolvent.
We begin by defining the truncated generator:
let $P_{\leq k}$ be the projection onto $\Gamma L^2_{\leq k}\eqdef\bigoplus_{n=0}^k\mathscr{H}_n$, i.e. onto the first $k$ chaoses and $\mathcal L_k^N=P_{\leq k}\mathcal L^N P_{\leq k}$.
Let $\mathfrak h^{N,k}$ be the solution to the truncated generator equation
\begin{equ}\label{eq:truncGen}
	\left(\lambda-\mathcal L^N_k\right)\mathfrak h^{N,k}=\mathfrak n^N \, .
\end{equ}
The following lemma was first proved in \cite[Lemma 2.1]{LQSY2004Superdiffusivity}.
\begin{lemma}\label{le:yau}
	Let $\lambda>0$. Then for every $k,N\in\mathbb N$ we have that
	\begin{equ}
		\left\langle \mathfrak{n}^N,\mathfrak h^{N,2k+1}\right\rangle
        \leq
        \left\langle \mathfrak{n}^N,\left(\lambda-\mathcal L^N\right)^{-1}\mathfrak{n}^N\right\rangle
        \leq
        \left\langle \mathfrak{n}^N,\mathfrak h^{N,2k}\right\rangle\,,
	\end{equ}
	furthermore both bounds monotonically converge to $\left\langle \mathfrak{n}^N,\left(\lambda-\mathcal L^N\right)^{-1}\mathfrak{n}^N\right\rangle$ as $k\to\infty$.
\end{lemma}

Equation (\ref{eq:truncGen}) written chaos-by-chaos takes the form
\begin{equs}\label{eq:system}
	(\lambda - \mathcal L_0) \mathfrak h^{N,k}_k &-\mathcal A_+^N\mathfrak h^{N,k}_{k-1} = 0 \, ,\\
	(\lambda - \mathcal L_0) \mathfrak h^{N,k}_{k-1} &-\mathcal A_+^N\mathfrak h^{N,k}_{k-2} -\mathcal A_-^N\mathfrak h^{N,k}_{k}= 0 \, ,\\
	& \dots \\
	(\lambda - \mathcal L_0) \mathfrak h^{N,k}_{2} &-\mathcal A_+^N\mathfrak h^{N,k}_{1} -\mathcal A_-^N\mathfrak h^{N,k}_{3} = \mathfrak n^N \, , \\
	(\lambda - \mathcal L_0) \mathfrak h^{N,k}_{1} & -\mathcal A_-^N\mathfrak h^{N,k}_{2} = 0 \, .
\end{equs}
This system of equations can be solved iteratively starting from the top, which leads to the following definition.
\begin{definition}
    For $k\ge 3$ we define the operators
	\begin{equ}
		\mathcal H_2^N \eqdef 0 \qquad \text{and} \qquad \mathcal H_k^N = \left(\mathcal A_+^N\right)^* \left(\lambda - \mathcal L_0 + \mathcal H_{k-1}^N\right)^{-1} \mathcal A_+^N \, .
	\end{equ}
\end{definition}
These operators are defined in an analogous way to the operators of the same name in \cite{CET2023stationary} and thus share some basic properties.
\begin{lemma}[Lemma 3.2 from \cite{CET2023stationary}]
	For $k\geq 3$, the operators $\mathcal H_k$ are positive definite and such that for all $n\in\mathbb N$ the operator $\mathcal H_k$ maps the $n$-th chaos into the $n$-th chaos.
\end{lemma}

Solving the system of equations (\ref{eq:system}) we obtain
\begin{equ}\label{eq:SOEsol}
	\mathfrak h^{N,k}_2=\left((\lambda-\mathcal L_0)+\mathcal H^N_k-\mathcal A_+^N(\lambda-\mathcal L_0)^{-1}\mathcal A_-^N\right)^{-1}\mathfrak n^N\,.
\end{equ}
For the operator $-\mathcal A_+^N(\lambda-\mathcal L_0)^{-1}\mathcal A_-^N$ notice the following:
consider the subspace $\mathscr V$ of $\Gamma L^2$ generated by $\psi$ supported only on $k_{1:n}$ which satisfy $\sum_{i=1}^nk_i=0$ (for arbitrary $n$).
The operators $\mathcal A_+^N,\mathcal A_-^N$  and $\mathcal L_0$ all map $\mathscr V$ into $\mathscr V$ and the orthogonal complement of $\mathscr V$ into the orthogonal complement of $\mathscr V$.
Additionally, $\mathcal A_-^N$ vanishes on the intersection of $\mathscr V$ and $\Gamma L^2_2$.
Since $\mathfrak n^N$ is in $\mathscr V\cap\Gamma L^2_2$ this implies
\begin{equ}\label{eq:dropAplusAminus}
	\left((\lambda-\mathcal L_0)+\mathcal H^N_k-\mathcal A_+^N(\lambda-\mathcal L_0)^{-1}\mathcal A_-^N\right)^{-1}\mathfrak n^N=\left((\lambda-\mathcal L_0)+\mathcal H^N_k\right)^{-1}\mathfrak n^N\,.
\end{equ}

The following lemma summarises the result of this subsection.
\begin{lemma}\label{le:innerproduct}
	For all $\lambda>0$, $N\in\mathbb N$ and $k\ge2$ it holds that 
	\begin{equ}
		\left\langle \mathfrak{n}^N,\mathfrak h^{N,k}\right\rangle=\left\langle\mathfrak n^N,\left((\lambda-\mathcal L_0)+\mathcal H^N_k\right)^{-1}\mathfrak n^N\right\rangle \, .
	\end{equ}
\end{lemma}
\begin{proof}
	This follows from (\ref{eq:truncGen}), (\ref{eq:SOEsol}) and (\ref{eq:dropAplusAminus}).
\end{proof}

So all that remains is to estimate the operators $\mathcal{H}_k$'s.

\section{Iterative Estimates}\label{se:estimates}

In this section we set up the iterative estimation scheme for the operators $\mathcal H_k$.
In order to do so, we first need to give some definitions.

The skew Laplacian $\mathcal L^{\mathfrak w}_0$ is the linear operator whose action on Fock space is given, for every $\varphi \in \Gamma L^2_n$, by
\begin{equ}\label{eq:skewed_laplacian}
	\widehat{\mathcal L_0^{\mathfrak w} \varphi}(k_{1:n})\eqdef-\tfrac12 (\mathfrak w \cdot k)^2_{1:n}\widehat \varphi(k_{1:n}) \, , \qquad \text{where} \qquad (\mathfrak w\cdot k)_{1:n}^2\eqdef \sum_{i=1}^n(\mathfrak w \cdot k_i)^2 \, .
\end{equ}
The sequence of the exponents of the logarithm in the upper and lower bounds is defined recursively by 
\begin{equ}
	\theta_2 \eqdef 0 \qquad \text{and} \qquad \theta_{k+1} = 1 - \frac{\theta_k}{2} \qquad  \text{for every $k\ge 3$} \, ,
\end{equ}
and admits the close formula
\begin{equ} \label{eq:explicit_alpha}
	\theta_k = \frac{2}{3}\left(1 - \left(\frac{-1}{2}\right)^{k-2}\right) \qquad \text{for every $k \ge 2$} \, .
\end{equ}
We now introduce some elementary functions, that morally approximate a logarithm to the power two third. Let $k, N, n \in \mathbb N$ and $\delta \in (0,1)$. For $k\ge2$, $x\in (0,\infty)$ and $z\in(1,\infty)$, we set
\begin{equ}\label{eq:log}
	\mathrm L(x,z) \eqdef \log(1+x^{-1})+z \, , \quad  \mathrm L_k(x,z) \eqdef \left(\mathrm L(x,z)\right)^{\theta_k} \, , \quad \mathrm L_k^N(x,z) \eqdef \mathrm L_k\left(\frac{x}{N^2},z\right) \, .
\end{equ}
The functions above are accompanied by polynomial coefficients, that morally correspond to errors made in the estimates. For $k\ge 1$, those are given by
\begin{equ} \label{eq:z_and_f}
	z_k(n)\eqdef K (n+k)^{\frac{9}{2}+\frac{3}{2}\delta} \qquad \text{and} \qquad f_k(n)\eqdef 3 (z_k(n))^{\frac{2}{3}} \, ,
\end{equ}
where $K$ is a sufficiently large positive constant depending on $\lvert \mathfrak w \rvert$. 
\begin{remark}
   The exact dependence of $K$ on $\lvert \mathfrak w \rvert$ is not important for us. However, following the proof, it is not difficult to check that the lower bound that $K$ must satisfy is of the kind $a\big(\lvert \mathfrak w \rvert \vee \frac{1}{\lvert \mathfrak w\rvert}\big)^b$ for some $a,b>0$.
\end{remark}
We also note the trivial identities
\begin{equ}[eq:ex_iter_chaos]
	z_k(n+1)=z_{k+1}(n) \qquad \text{and} \qquad f_k(n+1)=f_{k+1}(n) \, .
\end{equ}
Moreover, in the proofs we use the additional notation 
\begin{equs} \label{eq:tildeGamma}
	&\tilde\Gamma  \eqdef \tilde\Gamma(\ell,m,k_{2:n}) \eqdef \tfrac12\left(\vert\ell\rvert^2+\lvert m\rvert^2+\lvert k_{2:n} \rvert^2\right) \, ,\\
	&\tilde\Gamma^{\mathfrak w}  \eqdef \tilde\Gamma^{\mathfrak w}(\ell,m,k_{2:n})\eqdef \tfrac12\left((\mathfrak w \cdot\ell)^2+(\mathfrak w\cdot m)^2+(\mathfrak w\cdot k)^2_{2:n}\right) \, ,
\end{equs}
where $k_{2:n}$ means $k_2, \dots, k_n$ and $\ell$, $m$, $k_2,\dots,k_n \in \mathbb Z^2$ are Fourier modes. This is coherent with the notation used in \cite[Section 3]{CET2023stationary}.
By the Cauchy-Schwarz inequality, the symbols above can be compared as follows:
\begin{equ}\label{eq:tGamma_comp}
	0\le\tilde\Gamma^{\mathfrak w} \le \lvert \mathfrak w \rvert^2\tilde \Gamma \, .
\end{equ}
We are finally ready to give the definitions of the operators used for the iterative bounds.
\begin{definition}
	For $\lambda>0$ and $k\ge 2$
	\begin{equ}
		\mathcal S_k^N \eqdef
		\begin{cases}
			f_k(\mathcal N) \mathrm L^N_k\left(\lambda - \mathcal L_0,z_k(\mathcal N)\right) \qquad                                       & \text{if $k$ is odd,}  \\
			\frac{1}{f_k(\mathcal N)} \left[\mathrm L^N_k(\lambda - \mathcal L_0,z_k(\mathcal N)) - f_k(\mathcal N)\right] \qquad & \text{if $k$ is even,}
		\end{cases}
	\end{equ}
 where $\mathcal N$ is the number operator, acting on $\varphi\in\Gamma L_n$ by $\mathcal N\varphi=n\varphi$ for each $n\in\mathbb N$ and $\lambda$ is the Laplace variable.
\end{definition}

We will use the following (quite standard) partial ordering of operators:
\begin{definition}\label{de:def_ineq_operators}
	Given two self-adjoint operators $\mathcal A$ and $\mathcal B$ on $\Gamma L^2$,
	\begin{equ}
		\mathcal A \le \mathcal B
		\quad \Leftrightarrow \quad
		\forall \,n \, \,\forall \, \varphi \in \Gamma L^2_n \quad
        \langle \mathcal A \varphi, \varphi \rangle \le \langle \mathcal B \varphi, \varphi \rangle
		\quad \Leftrightarrow \quad
        \mathcal B - \mathcal A \geq 0\,,
	\end{equ}
    where the last statement is taken to mean that $\mathcal B - \mathcal A$ is a positive operator.
\end{definition}
For this partial ordering the following well-known lemma holds
\begin{lemma}
	For any two operators self-adjoint $\mathcal A$ and $\mathcal B$ on $\Gamma L^2$ it holds that
	\begin{equ}
		0<\mathcal A\leq \mathcal B
        \quad \Leftrightarrow \quad
        0<\mathcal B^{-1}\leq \mathcal A^{-1}\,.
	\end{equ}
\end{lemma}

We can now state the bounds on the operators $\mathcal{H}_k$.
\begin{theorem}[iterative bounds]
	\label{th:iterative_bounds}
	For every $\delta \in (0,1)$ and for  every $k\in \mathbb Z$, $k\ge 0$ we have 
	\begin{equs}
		\mathcal H_{2k+3} & \le c_{2k+3} \left((-\mathcal L_0^{\mathfrak w})  \mathcal S_{2k+3}+ f_{2k+3}(\mathcal N)(-\mathcal L_0)\right)  \, ,\label{eq:iter_UB} \\
		\mathcal H_{2k+2} & \ge c_{2k+2} \left((-\mathcal L_0^{\mathfrak w}) \mathcal S_{2k+2} - \frac{1}{(\mathcal N +k)^{1+\delta}}(-\mathcal L_0)\right), \label{eq:iter_LB}
	\end{equs}
	where the constants $c_{2k+1}$ and $c_{2k+2}$ are defined recursively by setting, for $k\ge 1$,
	\begin{equ}
	    c_2 \eqdef \frac{1}{\pi(\lvert \mathfrak w \rvert^2 \vee 1)} \, ,
        \quad
        c_{2k+1} = \frac{3}{2\pi\lvert \mathfrak w\rvert}\frac{\left(1+\frac{1}{2k^{1+\delta}}\right)}{c_{2k}} > 1 \, ,
		\quad 
		c_{2k+2} =\frac{3}{2\pi\lvert \mathfrak w\rvert} \frac{\left(1- \frac{1}{2k^{1+\delta}}\right)}{\left(1+\frac{1}{2k^{1+\delta}}\right) c_{2k+1}} <1
	\end{equ}
    and the $\delta$ explicitly appearing in~(\ref{eq:iter_LB}) is the same as the one used for defining the $\mathcal S_k$'s and the $c_k$'s.
\end{theorem}
Note that compared to the structure in previous works (\cite{CHSTDiffusionCurl}, \cite{CET2023stationary}), the upper and lower bounds are now split into a part multiplying the skew Laplacian $\mathcal L_0^{\mathfrak w}$ and a part multiplying the full Laplacian $\mathcal L_0$.
The second part is used to estimate the off-diagonals, i.e. we do not absorb them into the diagonal terms, but keep them separate.

\begin{remark}
	The two inequalities on odd and even terms of the sequence $(c_k)_{k\ge2}$ can be checked by induction, separately on odd and even terms, after distinguishing the two cases $\lvert \mathfrak w \rvert > 1 $ and $\lvert \mathfrak w \rvert < 1$. Also, note that $\sum_k \frac{1}{2k^{1+\delta}}$ is summable. This implies that, as $k\to\infty$, $c_{2k+1}$ and $c_{2k+2}$ converge to two positive and finite limits, respectively larger and smaller than $1$.
\end{remark}

We now proceed with some preliminary Lemmas, which will be used to prove Theorem~\ref{th:iterative_bounds}.

\begin{lemma}[Decomposition in diagonal and off-diagonal terms]
	\label{le:dec_diag_off}
	Let $\mathcal Z$ be a diagonal operator on $\Gamma L^2$ with Fourier multiplier $\zeta = (\zeta_n)_{n\in \mathbb N}$. Then, for every $\varphi \in \Gamma L^2_n$, the following decomposition holds:
	\begin{equ}
		\left\langle \left(\mathcal A_+^N\right)^*\mathcal Z \mathcal A^N_+ \varphi,\varphi\right\rangle
        =
        \left\langle \left(\mathcal A_+^N\right)^*\mathcal Z \mathcal A^N_+\varphi,\varphi\right\rangle_{\mathrm{Diag}} +\sum_{i=1}^2 \left\langle \left(\mathcal A_+^N\right)^*\mathcal Z \mathcal A^N_+\varphi,\varphi\right\rangle_{\mathrm{off}_i} \, ,
	\end{equ}
	where the diagonal terms, given by the first summand, are defined as
	\begin{equ}[eq:diag]
		\left\langle \left(\mathcal A_+^N\right)^*\mathcal Z \mathcal A^N_+\varphi,\varphi\right\rangle_{\mathrm{Diag}}\eqdef \frac{n!\, n}{2\pi^2} \sum_{k_{1:n}}(\mathfrak w\cdot k_1)^2\overline{\widehat\varphi(k_{1:n})}{\widehat\varphi(k_{1:n})}\!\!\!\sum_{\ell+m=k_1}  \mathbb J^N_{\ell,m} \zeta_{n+1}(\ell,m,k_{2:n}) \, ,
	\end{equ}
	while the off-diagonal terms of type $1$ and $2$ are respectively given by
	\begin{multline}\label{eq:off_diag1}
		\left\langle \left(\mathcal A_+^N\right)^*\mathcal Z \mathcal A^N_+\varphi,\varphi\right\rangle_{\mathrm{off}_1}  \eqdef \, \frac{n! \, n (n-1)}{\pi^2}\sum_{k_{1:n+1}} (\mathfrak w\cdot (k_1+k_2))  (\mathfrak w\cdot (k_1+k_3))  \times\\
		\times
		\overline{\widehat{\varphi}(k_1+k_2,k_3,k_{4:n+1})}{\widehat\varphi(k_1+k_3,k_2,k_{4:n+1})} \mathbb J^N_{k_1,k_2}\mathbb J^N_{k_1,k_3} \zeta_{n+1}(k_{1:n+1})
	\end{multline}
	and
	\begin{multline}\label{eq:off_diag2}
		\left\langle \left(\mathcal A_+^N\right)^*\mathcal Z \mathcal A^N_+\varphi,\varphi \right\rangle_{\mathrm{off}_2}  \eqdef  \, \frac{n! \, n (n-1) (n-2)}{4\pi^2}\sum_{k_{1:n+1}} (\mathfrak w\cdot (k_1+k_2))  (\mathfrak w\cdot (k_3+k_4)) \times\\
		\times \overline{\widehat \varphi(k_1+k_2,k_{3:4},k_{5:n+1})}\widehat \varphi(k_3+k_4,k_{1:2},k_{5:n+1})\mathbb J^N_{k_1,k_2}\mathbb J^N_{k_3,k_4} \zeta_{n+1}(k_{1:n+1})\, .
	\end{multline}
\end{lemma}
The above decomposition is the same as the one used in the proof of \cite[Lemma 2.5]{cannizzaro2023gaussian} and we refer to \cite[Lemma 3.6]{CET2023stationary} for the combinatorics needed for the exact expressions of the multiplicative factors in front of the sums.

The following two lemmas will be used to bound the diagonal and off-diagonal parts respectively by estimating sums over only one $\mathbb Z^2$-valued variable.
\begin{lemma}\label{le:control_diag}
	Let $\mathcal Z_1$ and $\mathcal Z_2$ be two positive  operators on $\Gamma L^2$, diagonal both in chaos and in Fourier, with Fourier multipliers $\zeta^i = (\zeta^i_n)_{n\in \mathbb N}$, for $i=1,2$. If for every $n\in \mathbb N$ and for every $k_{1:n} \in \mathbb Z^{2n}$
	\begin{equ} \label{eq:hyp_control_diag}
		\sum_{\ell+m=k_1}\mathbb J^N_{\ell,m}\zeta^1_{n+1}(\ell,m,k_{2:n})  \le \zeta^2_n(k_{1:n}) \, ,
	\end{equ}
	then for every $\varphi \in \Gamma L^2_n$
	\begin{equ}
		\left\langle \left(\mathcal A_+^N\right)^* \mathcal Z_1 \mathcal A^N_+ \varphi, \varphi \right\rangle_{\mathrm{Diag}} \le \frac{1}{\pi^2}\left\langle (-\mathcal L_0^{\mathfrak w}) \mathcal Z_2 \varphi, \varphi \right\rangle \, .
	\end{equ}
	Moreover, a reverse inequality in the assumption implies a reverse inequality in the result.
\end{lemma}
\begin{proof}
    Recalling the expressions of $\mathcal A_+^N$, $\left(\mathcal A_+^N\right)^*$ and $\mathcal L_0^{\mathfrak w}$, given in Lemma~\ref{le:generator} and in (\ref{eq:skewed_laplacian}) respectively, and using hypothesis (\ref{eq:hyp_control_diag}), we obtain:
	\begin{equs}
		&\left\langle \left(\mathcal A_+^N\right)^* \mathcal Z_1 \mathcal A^N_+ \varphi, \varphi \right\rangle_{\mathrm{Diag}} = \frac{n!\, n}{2\pi^2} \sum_{k_{1:n}}(\mathfrak w\cdot k_1)^2\lvert\widehat\varphi(k_{1:n})\rvert^2\!\!\!\sum_{\ell+m=k_1}\mathbb J^N_{l,m}\zeta^1_{n+1}(\ell,m,k_{2:n}) \le \\
		& \le \frac{n!\, n}{2\pi^2} \sum_{k_{1:n}}(\mathfrak w\cdot k_1)^2\lvert\widehat\varphi(k_{1:n})\rvert^2\zeta^2_n(k_{1:n}) = \frac{n!}{\pi^2} \sum_{k_{1:n}}(-\mathcal L_0^{\mathfrak w})\lvert\widehat\varphi(k_{1:n})\rvert^2\zeta^2_n(k_{1:n}) = \\
		& = \frac{1}{\pi^2}\left\langle (-\mathcal L_0^{\mathfrak w}) \mathcal Z_2 \varphi,\varphi \right\rangle \, .
	\end{equs}
	The reverse inequality follows by repeating the exact same steps above with the inequality in the other direction.
\end{proof}

\begin{lemma}\label{le:control_offdiag}
	Let $\mathcal Z_1$ and $\mathcal Z_2$ be two positive  operators on $\Gamma L^2$, diagonal both in chaos and in Fourier, with Fourier multipliers $\zeta^i = (\zeta^i_n)_{n\in \mathbb N}$, for $i=1,2$.

	If for every $n\in \mathbb N$ and for every $k_{1:n} \in \mathbb Z^{2n}$
	\begin{equ}
		\lvert k_2\rvert\sum_{\ell+m = k_1} \frac{1}{\lvert m \rvert}\mathbb J^N_{l,m} \zeta^1_{n+1}(\ell,m,k_{2:n})
        \le
        \zeta^2_n(k_{1:n}) \, ,
	\end{equ}
	then for every $\varphi \in \Gamma L^2_n$
	\begin{equ} \label{eq:control_offdiag_1}
		\left\lvert\left\langle \left(\mathcal A_+^N\right)^* \mathcal Z_1 \mathcal A^N_+ \varphi, \varphi \right\rangle_{\mathrm{off_1}}\right\rvert
        \le
        \frac{2\lvert \mathfrak w \rvert^2 (n-1)}{\pi^2} \left\langle (-\mathcal L_0) \mathcal Z_2 \varphi, \varphi \right\rangle \, .
	\end{equ}

	If for every $n\in \mathbb N$ and for every $k_{1:n} \in \mathbb Z^{2n}$
	\begin{equ}
		\sum_{\ell+m=k_1}\mathbb J^N_{\ell,m} \frac{\zeta^1_{n+1}(\ell, m,k_{2:n})}{\lvert \ell \rvert \lvert m \rvert}
        \le
        \frac{1}{\lvert k_1\rvert \sqrt{\lvert k_{1:n} \rvert^2}}\zeta^2_n(k_{1:n}) \, ,
	\end{equ}
	then for every $\varphi \in \Gamma L^2_n$
	\begin{equ}\label{eq:control_offdiag_2}
		\left\lvert\left\langle \left(\mathcal A_+^N\right)^* \mathcal Z_1 \mathcal A^N_+ \varphi, \varphi \right\rangle_{\mathrm{off_2}}\right\rvert
        \le
        \frac{\lvert \mathfrak w \rvert^2  (n-1) (n-2)}{\pi^2} \left\langle (-\mathcal L_0) \mathcal Z_2 \varphi, \varphi \right\rangle \, .
	\end{equ}
\end{lemma}
\begin{proof}
	Following ideas of \cite{CET2023stationary}, we define
	\begin{equ}
		\forall n\in \mathbb N \qquad \Phi(k_{1:n}) \eqdef \prod_{i=1}^{n}\lvert k_i \rvert \lvert \widehat\varphi(k_{1:n})\rvert \, . \label{eq:large_phi}
	\end{equ}

	We start with the estimate for the off-diagonal terms of type $1$, i.e. \eqref{eq:control_offdiag_1}.
	In order to have more easily readable expressions, we give a name to the coefficient (its exact expression is only used at the end of the proof):
	\begin{equ}
		c_{\mathrm{off}_1}(n) \eqdef \frac{\lvert \mathfrak w \rvert^2 n! \, n (n-1)}{\pi^2} \, .
	\end{equ}
	By expanding the left-hand side of~(\ref{eq:control_offdiag_1}), applying the Cauchy-Schwarz inequality and using definition~(\ref{eq:large_phi}), we get
	\begin{equs}
		&\left\lvert \left\langle \mathcal Z \mathcal A^N_+\varphi,\mathcal A^N_+\varphi\right\rangle_{\mathrm{off}_1} \right\rvert
		\stackrel{\text{\scriptsize{C-S}}}{\le}
		c_{\mathrm{off}_1}(n) \sum_{k_{1:n+1}} \lvert k_1+k_2 \rvert \lvert k_1+k_3 \rvert  \times \\
		&\qquad \qquad \times  \lvert \widehat{\varphi}(k_1+k_2,k_3,k_{4:n+1})\rvert \lvert \widehat\varphi(k_1+k_3,k_2,k_{4:n+1})\rvert \mathbb J^N_{k_1,k_2}\mathbb J^N_{k_1,k_3} \zeta^1_{n+1}(k_{1:n+1})  \\
		&= c_{\mathrm{off}_1}(n)\sum_{k_{1:n+1}} \frac{\Phi(k_1+k_2,k_3,k_{4:n+1}) \Phi(k_1+k_3,k_2,k_{4:n+1})}{\lvert k_2 \rvert \lvert k_3 \rvert \prod_{i=4}^{n+1}\lvert k_i \rvert^2} \mathbb J^N_{k_1,k_2}\mathbb J^N_{k_1,k_3} \zeta^1_{n+1}(k_{1:n+1}) \, . \\
	\end{equs}
	We now recall the elementary inequality $\lvert ab \rvert \le a^2/2 +b^2/2$, true for any $a,b \in \mathbb R$, and apply it with the choice $a=\Phi(k_1+k_2,k_3,k_{4:n+1})\mathbb J^N_{k_1,k_2}$ and $b$ analogous. By symmetry, the second addend that we obtain by this procedure is actually equal to the first one, and so we obtain the upper bound
	\begin{equ}
		c_{\mathrm{off}_1}(n) \sum_{k_{1:n+1}} \frac{(\Phi(k_1+k_2,k_3,k_{4:n+1}))^2}{\lvert k_2 \rvert \lvert k_3 \rvert \prod_{i=4}^{n+1}\lvert k_i \rvert^2} \mathbb J^N_{k_1,k_2} \zeta^1_{n+1}(k_{1:n+1}) \, .
	\end{equ}
	Expanding the definition of $\Phi$, applying the change of variables $k_{1:n+1}\mapsto (\ell,m,k_{2:n})$ and finally using the hypothesis gives the desired upper bound:
	\begin{equs}
		&  c_{\mathrm{off}_1}(n) \sum_{k_{1:n+1}} \frac{\lvert k_1+k_2 \rvert^2 \lvert k_3 \rvert \lvert \widehat{\varphi}(k_1+k_2,k_3,k_{4:n+1})\rvert^2}{\lvert k_2 \rvert} \mathbb J^N_{k_1,k_2} \zeta^1_{n+1}(k_{1:n+1})\\
		& = c_{\mathrm{off}_1}(n) \sum_{\ell,m,k_{2:n}} \frac{\lvert \ell+m \rvert^2 \lvert k_2 \rvert \lvert \widehat{\varphi}(\ell+m,k_{2:n})\rvert^2}{\lvert m \rvert} \mathbb J^N_{l,m} \zeta^1_{n+1}(\ell,m,k_{2:n})\\
		& = c_{\mathrm{off}_1}(n)\sum_{k_{1:n}}\lvert \widehat{\varphi}(k_1,k_{2:n})\rvert^2 \lvert k_1 \rvert^2 \lvert k_2 \rvert \sum_{\ell+m = k_1}\mathbb J^N_{l,m}\frac{\zeta^1_{n+1}(\ell,m,k_{2:n})}{\lvert m \rvert} \\
		& \le \frac{\lvert \mathfrak w \rvert^2 n! \, n (n-1)}{\pi^2}\sum_{k_{1:n}}\lvert \widehat{\varphi}(k_1,k_{2:n})\rvert^2 \lvert k_1 \rvert^2 \zeta^2_{n}(k_{1:n}) = \frac{2\lvert \mathfrak w \rvert^2 (n-1)}{\pi^2}\langle (-\mathcal L_0) \mathcal Z_2 \varphi, \varphi\rangle \, ,
	\end{equs}
	where the factor $n$ was absorbed in the definition of $(-\mathcal L_0)$ (recall~(\ref{eq:gen_Fourier})).

	We now prove the statement about the off-diagonal terms of type 2, i.e. \eqref{eq:control_offdiag_2}. We set
	\begin{equ}
		c_{\mathrm{off}_2}(n) \eqdef \frac{\lvert \mathfrak w \rvert^2 n! \, n (n-1) (n-2)}{4\pi^2} \,
	\end{equ}
	and follow the same steps already used for the off-diagonal terms of type 1, even though the expressions to which we apply them now are slightly different. More precisely, we consider the left-hand side of~(\ref{eq:control_offdiag_2}), use the Cauchy-Schwarz inequality, $ab \le a^2/2 + b^2/2$ with $a = \Phi(k_1+k_2,k_{3:4},k_{5:n+1})\mathbb J^N_{k_1,k_2}$ and $b$ analogous and the change of variables $k_{1:n+1} \mapsto (\ell,m,k_{2:n})$. Overall, this gives the upper bound
	\begin{equ}
		\left\lvert \left\langle \mathcal Z \mathcal A^N_+ \varphi,\mathcal A^N_+ \varphi \right\rangle_{\mathrm{off}_2} \right\rvert
        \le
        c_{\mathrm{off}_2}(n)\sum_{k_{1:n}} \lvert \widehat \varphi (k_1,k_{2:n})\rvert^2\lvert k_1\rvert^2 \lvert k_2 \rvert \lvert k_3 \rvert \sum_{\ell+m=k_1}\mathbb J^N_{\ell,m} \frac{\zeta^1_{n+1}(\ell, m,k_{2:n})}{\lvert \ell \rvert \lvert m \rvert} \, .
	\end{equ}
	Finally, by applying the hypothesis of the Lemma to innermost sum of the expression above, we obtain the upper bound
	\begin{equ}\label{eq:offterm}
		\frac{\lvert \mathfrak w \rvert^2 n! \, n (n-1) (n-2)}{4\pi^2}\sum_{k_{1:n}} \lvert \widehat \varphi (k_1,k_{2:n})\rvert^2 \frac{\lvert k_{1}\rvert \lvert k_{2}\rvert \lvert k_{3}\rvert}{\sqrt{\lvert k_{1:n} \rvert^2}} \zeta^2_{n}(k_{1:n}) \, .
	\end{equ}
    Using $\lvert k_3\rvert\leq\sqrt{\lvert k_{1:n}\rvert}$ the sum above is upper bounded by
    \begin{equ}
        \sum_{k_{1:n}}\lvert k_{1}\rvert \lvert k_{2}\rvert \lvert \widehat \varphi (k_{1:n})\rvert^2  \zeta^2_{n}(k_{1:n})\,.
    \end{equ}
    Since both $\widehat\varphi$ and $\zeta_n^2$ are symmetric, we can replace $\lvert k_{1}\rvert \lvert k_{2}\rvert$ in this sum by $\frac{1}{\binom{n}{2}} \sum_{i\neq j} \lvert k_i \rvert  \lvert k_j \rvert$, which we further estimate by
    \begin{equ}
        \frac{1}{\binom{n}{2}} \sum_{i\neq j} \lvert k_i \rvert  \lvert k_j \rvert
        \le \frac{1}{2\binom{n}{2}} \sum_{i\neq j} \lvert k_i \rvert^2+  \lvert k_j\rvert^2
        =\frac{2}{n}\sum_{i=1}^n\lvert k_i\rvert^2\,.
    \end{equ}
    Doing so we obtain that~\eqref{eq:offterm} is bounded by the right-hand side of~\eqref{eq:control_offdiag_2}.
    This concludes the proof.
\end{proof}

We are now ready to prove Theorem~\ref{th:iterative_bounds}. The proof is written separately for the upper and lower bounds.

\begin{proof}[Proof of Theorem~\ref{th:iterative_bounds}, inequality~(\ref{eq:iter_LB})]
	We proceed by induction. In this first part of the proof we show that the $(2k+2)$-th lower bound holds assuming that the $(2k+1)$-th upper bound does. In the next part, instead, we will assume the $(2k+2)$-th lower bound and prove that the $(2k+3)$-th upper bound holds.

	As for all proofs by induction, we need an initial step. We take this to be the lower bound for $k=0$. More precisely, this consists in showing
	\begin{equ}
		0 = \mathcal H_2 \ge c_2 \left((-\mathcal L_0^{\mathfrak w})\frac{1}{f_2(\mathcal N)}\left(1-f_2(\mathcal N)\right)-\frac{1}{(\mathcal N +2)^{1+\delta}}(-\mathcal L_0)\right) \, .
	\end{equ}
	Since $f_2(n)>1$ for every $n\in \mathbb N$, the right-hand side above is negative and thus the inequality holds for any arbitrary choice of $c_2>0$ (uniformly in $n,k\ge1$), so we may as well choose the one given in the statement of the theorem:
	\begin{equ}
		c_2 \eqdef \frac{1}{\pi(\lvert \mathfrak w \rvert^2 \vee 1)} \, .
	\end{equ}

	We now proceed to the inductive argument. Let $k\ge 1$.
	Assume by induction that~(\ref{eq:iter_UB}) holds for $k-1$, i.e. assume the upper bound stated for $\mathcal H_{2k+1}$. We want to prove~(\ref{eq:iter_LB}) for $k$, i.e. we want to prove the lower bound stated for $\mathcal H_{2k+2}$. We have:
	\begin{equs}
		\mathcal H_{2k+2} &= \left(\mathcal A_+^N\right)^*(\lambda - \mathcal L_0 +\mathcal H_{2k+1})^{-1}\mathcal A_+^N \\
		& \ge \left(\mathcal A_+^N\right)^*\left(\lambda - \mathcal L_0 +c_{2k+1} \left[(-\mathcal L_0^{\mathfrak w})  \mathcal S_{2k+1}+ f_{2k+1}(\mathcal N)(-\mathcal L_0)\right]\right)^{-1}\mathcal A_+^N \\
		& \eqdef \left(\mathcal A_+^N\right)^*\mathcal Z_{2k+1}\mathcal A_+^N \, ,
	\end{equs}
	where we used the last equation to define the operator $\mathcal Z_{2k+1}$. We observe that $\mathcal Z_{2k+1}$ is diagonal in Fourier and thus, consistently with the notation already used in the other lemmas of this section, we denote by $\zeta^{2k+1}$ its Fourier multiplier.

	By Remark~\ref{de:def_ineq_operators}, we set out to bound $\langle \left(\mathcal A_+^N\right)^*\mathcal Z_{2k+1}\mathcal A_+^N \varphi, \varphi \rangle$ from below, for $\varphi \in \Gamma L^2_n$.
	We recognize an expression of the type considered in Lemma~\ref{le:dec_diag_off}, \ref{le:control_diag}, and ~\ref{le:control_offdiag}. First of all, we use Lemma~\ref{le:dec_diag_off} to split the scalar product into diagonal and off-diagonal terms:
\begin{equ}\label{eq:split}
    \langle \left(\mathcal A_+^N\right)^*\mathcal Z_{2k+1}\mathcal A_+^N \varphi, \varphi \rangle=\langle \left(\mathcal A_+^N\right)^*\mathcal Z_{2k+1}\mathcal A_+^N \varphi, \varphi \rangle_{\mathrm{Diag}}+\langle \left(\mathcal A_+^N\right)^*\mathcal Z_{2k+1}\mathcal A_+^N \varphi, \varphi \rangle_{\mathrm{off}}
\end{equ}
Then we proceed to study them separately, starting with the diagonal ones.

	Recall the definitions of $\tilde\Gamma$ and $\tilde\Gamma^{\mathfrak w}$ given in~(\ref{eq:tildeGamma}). In order to apply Lemma~\ref{le:control_diag}, we need a bound on the sum 
	\begin{equ}\label{eq:horror1}
		\sum_{\ell+m=k_1}\frac{\mathbb J^N_{\ell,m} }{\lambda + \tilde\Gamma +c_{2k+1}f_{2k+2}\left[\tilde\Gamma^{\mathfrak w}\mathrm L^N_{2k+1}(\lambda + \tilde \Gamma,z_{2k+2})+\tilde\Gamma\right]} \, ,
	\end{equ}
	where we first used property~(\ref{eq:ex_iter_chaos}) to replace $z_{2k+1}(n+1)$ and $f_{2k+1}(n+1)$ with $z_{2k+2}(n)$ and $f_{2k+2}(n)$ respectively and then suppressed the argument of $z_{2k+1}$ and $f_{2k+1}$, as it is constant throughout. Inside the sum above we recognize the Fourier multiplier $\zeta^{2k+1}$ multiplied by the product of indicator functions $\mathbb J^N_{\ell,m}$, as by hypothesis of Lemma~\ref{le:control_diag}.

	The estimate of those kind of sums is carried out in Appendix~\ref{ap:replacement_lemmas}. However, before invoking it, we do one additional step and lower bound it by:
	\begin{equ}\label{eq:horror2}
		\frac{1}{c_{2k+1}f_{2k+2}\left(1+\frac{1}{f_{2k+2}}\right)}\sum_{\ell+m=k_1}\frac{\mathbb J^N_{\ell,m} }{\lambda + \tilde\Gamma +\tilde\Gamma^{\mathfrak w}\mathrm L^N_{2k+1}(\lambda + \tilde \Gamma,z_{2k+2})} \, ,
	\end{equ}
	where we multiplied $\lambda$ by $1+c_{2k+1}f_{2k+2}$ (a number larger than $1$), replaced the $c_{2k+1}f_{2k+2}$ factor of $\tilde \Gamma \mathrm L^N_{2k+1}$ by $1+c_{2k+1}f_{2k+2}$ and finally factored $1+c_{2k+1}f_{2k+2}$ out and used
	\begin{equ}
		1+c_{2k+1}f_{2k+2}
        =
        c_{2k+1}f_{2k+2}\left(1+\frac{1}{c_{2k+1}f_{2k+2}}\right)
        \le
        c_{2k+1}f_{2k+2}\left(1+\frac{1}{f_{2k+2}}\right) \, .
	\end{equ}
	As announced, we now apply Lemma~\ref{le:int_est} to the sum in~(\ref{eq:horror2}) to lower bound the whole expression~(\ref{eq:horror2}) by
	\begin{equ} \label{eq:intermediate}
		\frac{B_{2k+2}}{f_{2k+2}}\frac{\pi}{\lvert\mathfrak w \rvert} \left[A_{2k+2} \mathrm L^N_{2k+2}\left(\lambda+\lvert k_{1:n}\rvert^2,z_{2k+2}\right) - \frac{4}{3}(z_{2k+2})^{\theta_{2k+2}}\right] \, ,
	\end{equ}
	where, in order to increase readability, we set
	\begin{equs}
		A_{2k+2}(n)
        &\eqdef
        1-\left(\lvert \mathfrak w\rvert C_{\mathrm{Diag}}+2+\frac{3}{\lvert \mathfrak w \rvert}\right)\frac{1}{z^{\frac{\theta_{2k+1}}{2}}} \, ,\\
		B_{2k+2}(n)
        &\eqdef
        \frac{3}{2c_{2k+1}\left(1+\frac{1}{f_{2k+2}(n)}\right)} \, .\label{eq:B_{2k+2}}
	\end{equs}
    We now proceed with two additional steps.
	In the first one, observing that $\theta_{2k+2} \le \frac{2}{3}$ for every $k\ge 0$ and recalling the definition of $f_{2k+2}$ given in~(\ref{eq:z_and_f}), we estimate the additive error in the square brackets of~(\ref{eq:intermediate}) by 
	\begin{equ} \label{eq:how_to_lose_time}
        \frac{4}{3}(z_{2k+2})^{\frac{2}{3}}
        \le
        \frac{1}{2} f_{2k+2}
        \le
        \left(1- \frac{1}{2 k^{1+\delta}}\right) f_{2k+2}\, .
	\end{equ}
	In the second step, instead, we observe that $\frac{\theta_{2k+1}}{2}\ge\frac{1}{3}$ for every $k\ge 1$, so that
	\begin{equ}
        A_{2k+2}
        \ge
        1-\left(\lvert \mathfrak w\rvert C_{\mathrm{Diag}}+2+\frac{3}{\lvert \mathfrak w \rvert}\right)\frac{1}{(z_{2k+2})^{\frac{1}{3}}}
        \ge
        1- \frac{1}{2k^{1+\delta}} \, ,
	\end{equ}
	where the second inequality is true because of the lower bound $(z_{2k+2})^{\frac13}\ge2K^{\frac13}k^{1+\delta}$, assuming that $K$ in definition~(\ref{eq:z_and_f}) is large enough.
    
    Summing up, (\ref{eq:intermediate}) is lower bounded by 
    \begin{equ} \label{eq:final}
		\frac{B_{2k+2}}{f_{2k+2}}\frac{\pi}{\lvert\mathfrak w \rvert} \left(1- \frac{1}{2k^{1+\delta}} \right) \left(\mathrm L^N_{2k+2}\left(\lambda+\lvert k_{1:n}\rvert^2,z_{2k+2}\right) - f_{2k+2}\right)\, 
	\end{equ}
    and by using the above as hypothesis in Lemma~\ref{le:control_diag}, we obtain the following bound on the diagonal term of the scalar product:
	\begin{equ}
		\left\langle \left(\mathcal A_+^N\right)^*\mathcal Z_{2k+1}\mathcal A_+^N \varphi, \varphi \right\rangle_{\mathrm{Diag}}
        \ge
        \left\langle \frac{B_{2k+2}(\mathcal N)}{\pi \lvert \mathfrak w \rvert}  \left(1- \frac{1}{2 k^{1+\delta}}\right)(-\mathcal L_0^{\mathfrak w})\mathcal S_{2k+2}^N \varphi, \varphi \right\rangle
		\quad \forall \, \varphi \in \Gamma L^2_n\, .
	\end{equ}

	We now proceed to estimate the off-diagonal terms. We want to apply Lemma~\ref{le:control_offdiag}, whose hypotheses require estimates on two sums involving the Fourier multiplier $\zeta^{2k+1}$, one for the off-diagonal terms of type $1$ and one for the off-diagonal terms of type $2$.

	The sum that needs to be estimated for the off-diagonal terms of type $1$ is
	\begin{equ}\label{eq:offdiag1}
		\lvert k_2\rvert\sum_{\ell+m = k_1} \frac{1}{{\lvert m \rvert}} \frac{\mathbb J^N_{\ell,m} }{\lambda + \tilde\Gamma +c_{2k+1}f_{2k+2}\left[\tilde\Gamma^{\mathfrak w}\mathrm L^N_{2k+1}(\lambda + \tilde \Gamma,z_{2k+2})+\tilde\Gamma\right]} \, ,
	\end{equ}
	where, again, we first used property~(\ref{eq:ex_iter_chaos}) and then suppressed the variable $n$. 
	By dropping from the denominator of~(\ref{eq:offdiag1}) $\lambda + \tilde \Gamma$ and all other terms involving the Fourier modes $k_{\{1:n\}\setminus\{2\}}$, and by lower bounding $\mathrm L^N_{2k+1}$ by $0$, we obtain the upper bound
	\begin{equ}
		\frac{\lvert k_2\rvert}{c_{2k+1}f_{2k+2}}\sum_{\ell+m = k_1} \frac{1}{{\lvert m \rvert}} \frac{\mathbb J^N_{\ell,m} }{ \left(\lvert m \rvert^2+\lvert k_2 \rvert^2\right)} \lesssim \frac{1}{c_{2k+1}f_{2k+2}} \, ,
	\end{equ}
	where we estimated the sum by the corresponding integral and applied Lemma~\ref{le:arctan_parameter} with $\beta = \lvert k_2 \rvert^2$ and $\gamma=1$.

	Regarding the off-diagonal terms of type $2$, instead, we need to estimate the sum
	\begin{equs}\label{eq:offdiag2}
		&\sum_{\ell+m = k_1} \frac{1}{{\lvert \ell \rvert \lvert m \rvert}} \frac{\mathbb J^N_{\ell,m} }{\lambda + \tilde\Gamma +c_{2k+1}f_{2k+2}\left[\tilde\Gamma^{\mathfrak w}\mathrm L^N_{2k+1}(\lambda + \tilde \Gamma,z_{2k+2})+\tilde\Gamma\right]} \, .
	\end{equs}
    We now observe that the condition $\ell+m = k_1$ implies that at least one between $\ell$ and $m$ has norm larger than the one of $\tfrac12 k_1$. Using this to replace one $\lvert \ell\rvert^2$ by $\lvert \tfrac12 k_1 \rvert^2$ and by applying arguments analogous to the ones just used in the estimate of~(\ref{eq:offdiag1}), we obtain that~(\ref{eq:offdiag2}) is upper bounded by
	\begin{equ}
		\frac{16}{c_{2k+1}f_{2k+2}}\sum_{\ell+m = k_1} \frac{1}{{\lvert k_1 \rvert \lvert m \rvert}} \frac{\mathbb J^N_{\ell,m} }{\left(\lvert m \rvert^2 +\lvert k_{1:n}\rvert^2\right)} \lesssim  \frac{1}{c_{2k+1}f_{2k+2}} \frac{1}{\lvert k_1 \rvert \sqrt{\lvert k_{1:n}\rvert^2}} \, ,
	\end{equ}
	where we estimated the sum by the corresponding integral and applied Lemma~\ref{le:arctan_parameter} with $\beta = \lvert k_{1:n} \rvert^2$ and $\gamma=1$.

	Using Lemma~\ref{le:control_offdiag} with the estimates above as hypotheses, we conclude that
	\begin{equs}
		&\left\lvert\left\langle \left(\mathcal A_+^N\right)^*\mathcal Z_{2k+1}\mathcal A_+^N \varphi, \varphi \right\rangle_{\mathrm{Off}_1}\right\rvert + \left\lvert\left\langle \left(\mathcal A_+^N\right)^*\mathcal Z_{2k+1}\mathcal A_+^N \varphi, \varphi \right\rangle_{\mathrm{Off}_2}\right\rvert \\
		&\le C_{\mathrm{off}} \left\langle\left(\frac{2\lvert \mathfrak w \rvert^2 (\mathcal N-1)}{c_{2k+1}f_{2k+2}(\mathcal N)}+ \frac{\lvert \mathfrak w \rvert^2  (\mathcal N-1) (\mathcal N-2)}{c_{2k+1}f_{2k+2}(\mathcal N)}\right)(-\mathcal L_0)\varphi,\varphi\right\rangle\\
		&\le C_{\mathrm{off}} \left\langle\frac{\lvert \mathfrak w \rvert^2 \mathcal N^{\,2}}{c_{2k+1}f_{2k+2}(\mathcal N)}(-\mathcal L_0)\varphi,\varphi\right\rangle \, ,
	\end{equs}
	where $C_{\mathrm{off}}$ is an absolute constant independent of all variables at play (see~(\ref{eq:variables}) for more details).
	With this we have come to the last step of the proof of the iterative lower bounds. Recall that we were trying to estimate \eqref{eq:split} from below. In order to do so, we split it into diagonal and off-diagonal terms and we bounded each of them separately. We now put those estimates together:
	\begin{equs}
		& \left\langle \left(\mathcal A_+^N\right)^*\mathcal Z_{2k+1}\mathcal A_+^N \varphi, \varphi \right\rangle  \\
		& \ge \left\langle \left(\mathcal A_+^N\right)^*\mathcal Z_{2k+1}\mathcal A_+^N \varphi, \varphi \right\rangle_{\mathrm{Diag}} - \left\lvert \left\langle \left(\mathcal A_+^N\right)^*\mathcal Z_{2k+1}\mathcal A_+^N \varphi, \varphi \right\rangle_{\mathrm{off}_1}\right\rvert -\left\lvert\left\langle \left(\mathcal A_+^N\right)^*\mathcal Z_{2k+1}\mathcal A_+^N \varphi, \varphi \right\rangle_{\mathrm{off}_2}\right\rvert\\
		& \ge \left\langle\left[\frac{B_{2k+2}(\mathcal N)}{\pi \lvert \mathfrak w \rvert} \left(1- \frac{1}{2k^{1+\delta}}\right)(-\mathcal L_0^{\mathfrak w})\mathcal S_{2k+2}^N -\frac{C_{\mathrm{off}}\lvert \mathfrak w \rvert^2 \mathcal N^{\,2}}{c_{2k+1}f_{2k+2}(\mathcal N)}(-\mathcal L_0)\right] \varphi, \varphi \right\rangle   \\
		& \ge \left\langle\frac{B_{2k+2}(\mathcal N)\left(1- \frac{1}{2k^{1+\delta}}\right)}{\pi \lvert \mathfrak w \rvert} \left((-\mathcal L_0^{\mathfrak w})\mathcal S_{2k+2}^N - \frac{8 \pi C_{\mathrm{off}}\lvert \mathfrak w \rvert^3 \mathcal N^{\,2}}{3f_{2k+2}(\mathcal N)}(-\mathcal L_0)\right) \varphi, \varphi \right\rangle \, ,
	\end{equs}
	where we first used the triangular inequality and then the bounds established in proof. Finally, the last inequality is obtained by factoring out the coefficient of $(-\mathcal L_0^{\mathfrak w})\mathcal S_{2k+2}^N$ and upper bounding both $\left(1+\frac{1}{f_{2k+2}}\right)$ and $\left(1- \frac{1}{2k^{1+\delta}}\right)^{-1}$ by $2$.

	At this point, the proof is almost complete. We just need a few more estimates on the coefficients of the operators appearing in the scalar product above, so that it becomes exactly the one lower-bounding $\mathcal H_{2k+2}$ in the inequality~(\ref{eq:iter_LB}) of Theorem~\ref{th:iterative_bounds}.
	First of all, by recalling the definition of $f_{2k+2}$ given in (\ref{eq:z_and_f}) and taking $K$ large enough, we estimate the coefficient of $(-\mathcal L_0)$ by above by
	\begin{equs}
		\frac{8 \pi C_{\mathrm{off}}\lvert \mathfrak w \rvert^3 n^2}{3f_{2k+2}}
        =
        \frac{8 \pi C_{\mathrm{off}}\lvert \mathfrak w \rvert^3 n^2}{9 K^{2/3}(n+2k+2)^{3+\delta}}
        \le
        \frac{1}{(n+k)^{1+\delta}} \, .
	\end{equs}
	Finally, we consider the coefficient in front of the square bracket and observe that
	\begin{equ}
		\frac{B_{2k+2}\left(1- \frac{1}{2k^{1+\delta}}\right)}{\pi \lvert \mathfrak w \rvert}
		=
		\frac{3\left(1- \frac{1}{2k^{1+\delta}}\right)}{2\pi \lvert \mathfrak w \rvert c_{2k+1}\left(1+\frac{1}{f_{2k+2}}\right)}
		\ge
		c_{2k+2} \, ,
	\end{equ}
	simply by expanding the definition of $B_{2k+2}$ given at~(\ref{eq:B_{2k+2}}), lower bounding $f_{2k+2}$ by $2k^{1+\delta}$ and recalling the definition of $c_{2k+2}$ given in the statement of Theorem~\ref{th:iterative_bounds}. This concludes the proof.
\end{proof}

\begin{proof}[Proof of Theorem~\ref{th:iterative_bounds}, inequality~(\ref{eq:iter_UB})]
	In this second part of the proof we show that the $(2k+3)$-th upper bound holds assuming that the $(2k+2)$-th lower bound does. In the hope of making the reading easier, we note that the general structure of the two parts is similar.

	Let $k\ge 0$. Assume by induction that~(\ref{eq:iter_LB}) holds for $k$, i.e. assume the lower bound stated for $\mathcal H_{2k+2}$. We want to prove that also~(\ref{eq:iter_UB}) holds for $k$, i.e. we want to prove the upper bound stated for $\mathcal H_{2k+3}$. Then:
	\begin{equs}
		\mathcal H_{2k+3} &= \left(\mathcal A_+^N\right)^*\left(\lambda - \mathcal L_0 +\mathcal H_{2k+2}\right)^{-1}\mathcal A_+^N \\
		& \le \left(\mathcal A_+^N\right)^*\left(\lambda - \mathcal L_0 +c_{2k+2} \left[(-\mathcal L_0^{\mathfrak w}) \mathcal S_{2k+2} - \frac{1}{(\mathcal N +k)^{1+\delta}}(-\mathcal L_0)\right]\right)^{-1}\mathcal A_+^N \\
		& \eqdef \left(\mathcal A_+^N\right)^*\mathcal Z_{2k+2}\mathcal A_+^N \, ,
	\end{equs}
	where we the last equation defines $\mathcal Z_{2k+2}$. We denote by $\zeta^{2k+2}$ its Fourier multiplier.

	By Remark~\ref{de:def_ineq_operators}, our aim is to bound $\langle \left(\mathcal A_+^N\right)^* \mathcal Z_{2k+2} \mathcal A^N_+ \varphi, \varphi \rangle$ from above, for $\varphi \in \Gamma L^2_n$.
	We use~\ref{le:dec_diag_off} to split the scalar product into diagonal  and off-diagonal terms.
	We start by studying the diagonal ones. Recall the definitions of $\tilde\Gamma$ and $\tilde\Gamma^{\mathfrak w}$ given in~(\ref{eq:tildeGamma}). In order to apply Lemma~\ref{le:control_diag} we need a bound on
	\begin{equ}\label{eq:horror3}
		\sum_{\ell+m=k_1}\frac{\mathbb J^N_{\ell,m} }{\lambda + \tilde\Gamma +c_{2k+2}\left[\frac{\tilde\Gamma^{\mathfrak w}}{f_{2k+3}}\left(\mathrm L^N_{2k+2}(\lambda + \tilde \Gamma,z_{2k+3})-f_{2k+3}\right)-\frac{\tilde\Gamma}{(n+1+k)^{1+\delta}}\right]} \, ,
	\end{equ}
    where we first used property~(\ref{eq:ex_iter_chaos}) to replace $z_{2k+1}(n+1)$ and $f_{2k+1}(n+1)$ with $z_{2k+2}(n)$ and $f_{2k+2}(n)$ respectively and then suppressed the argument of $z_{2k+1}$ and $f_{2k+1}$, as it is constant throughout.
	The plan is to estimate this sum by using Lemma~\ref{le:int_est}, but before being able to do so we need to manipulate it a bit. While this was also the case for the proof of the iterative lower bounds, this time the process is a bit more involved, because not all addends in the denominator are positive. We start by expanding the denominator and applying inequality~(\ref{eq:tGamma_comp}):
	\begin{equs}
		&\sum_{\ell+m=k_1}\frac{\mathbb J^N_{\ell,m} }{\lambda + \tilde\Gamma +\frac{c_{2k+2}}{f_{2k+3}}\tilde\Gamma^{\mathfrak w}\,\mathrm L^N_{2k+2}(\lambda + \tilde \Gamma,z_{2k+3})-c_{2k+2}\tilde\Gamma^{\mathfrak w}-\frac{c_{2k+2}}{(n+1+k)^{1+\delta}}\tilde\Gamma} \\
		&\le\sum_{\ell+m=k_1}\frac{\mathbb J^N_{\ell,m} }{\lambda + \left(1-\lvert \mathfrak w \lvert^2c_{2k+2}-\frac{c_{2k+2}}{(n+1+k)^{1+\delta}}\right)\tilde\Gamma +\frac{c_{2k+2}}{f_{2k+3}}\tilde\Gamma^{\mathfrak w}\,\mathrm L^N_{2k+2}(\lambda + \tilde \Gamma,z_{2k+3})} \, . \qquad \label{eq:horror4}
	\end{equs}
	Now the goal is to factor out the coefficients of $\tilde \Gamma$ and of $\tilde\Gamma^{\mathfrak w}\,\mathrm L^N_{2k+2}$, in the same fashion in which expression~(\ref{eq:horror2}) was obtained.
	In order to be able to do this, we need some control on those coefficients. We start by estimating $c_{2k+2}$ as follows:
	\begin{equ}\label{eq:est_c2}
		\lvert \mathfrak w \rvert^2 c_{2k+2}
        =
        \lvert \mathfrak w \rvert^2 \frac{\prod_{j=1}^{k} \left(1-\frac{1}{2j^{1+\delta}}\right)}{\prod_{j=1}^{k}\left(1+\frac{1}{2j^{1+\delta}} \right)^2} \, c_2
        \le
        \lvert \mathfrak w \rvert^2\frac{1}{\pi (\lvert \mathfrak w \rvert^2 \vee 1)}
        \le
        \frac{1}{\pi} \, .
	\end{equ}
	Thus expression~(\ref{eq:horror4}) is upper bounded by the following:
	\begin{equ}\label{eq:horror5}
		\sum_{\ell+m=k_1}\frac{\mathbb J^N_{\ell,m} }{\lambda + \left(1-\frac{1}{\pi}-\frac{1}{\pi(n+1+k)^{1+\delta}}\right)\tilde\Gamma +\frac{c_{2k+2}}{f_{2k+3}}\tilde\Gamma^{\mathfrak w}\,\mathrm L^N_{2k+2}(\lambda + \tilde \Gamma,z_{2k+3})} \, .
	\end{equ}
	In particular, we observe that the coefficient of $\tilde \Gamma$ is positive. Moreover, by~(\ref{eq:est_c2}) above and for a large enough $K$, we have that $\frac{c_{2k+2}}{f_{2k+3}} \le 1-\frac{2}{\pi}$.
	This means that replacing the coefficient of $\tilde \Gamma$ by the one of  $\tilde\Gamma^{\mathfrak w}\,\mathrm L^N_{2k+2}$ gives an upper bound. Further multiplying $\lambda$ by $c_{2k+2}/f_{2k+3}$ and factoring out finally provides us with an upper bound of the kind we were looking for:
	\begin{equ}
		\frac{f_{2k+3}}{c_{2k+2}} \sum_{\ell+m=k_1}\frac{\mathbb J^N_{\ell,m} }{\lambda + \tilde\Gamma +\tilde\Gamma^{\mathfrak w}\,\mathrm L^N_{2k+2}(\lambda + \tilde \Gamma,z_{2k+3})} \, .
	\end{equ}

	We are finally ready to apply Lemma~\ref{le:int_est} to the sum in the expression above. This gives us the upper bound
	\begin{equ}
		\frac{3\pi}{2\lvert\mathfrak w \rvert}\left(1+\frac{\lvert \mathfrak w \rvert C_{\mathrm{Diag}}}{(z_{2k+3})^{\theta_{2k+3}}}\right) \mathrm L^N_{2k+3}\left(\lambda+\lvert k_{1:n} \rvert^2,z_{2k+3}\right) \, ,
	\end{equ}
	which, by choosing $K$ large enough, can be further upper bounded by
	\begin{equ} \label{eq:less_horror}
		\frac{3\pi}{2\lvert\mathfrak w \rvert}\left(1+\frac{1}{2(k+1)^{1+\delta}}\right) \mathrm L^N_{2k+3}\left(\lambda+\lvert k_{1:n} \rvert^2,z_{2k+3}\right) \, .
	\end{equ}

	Using the bound provided by expression~(\ref{eq:less_horror}) in the hypothesis of Lemma~\ref{le:control_diag} (and recalling the coefficient that was in front of the sum before invoking Appendix~\ref{ap:replacement_lemmas}), we obtain the following upper bound on the diagonal part:
	\begin{equs}
		&\left\langle \left(\mathcal A_+^N\right)^* \mathcal Z_{2k+2} \mathcal A^N_+ \varphi, \varphi \right\rangle_{\mathrm{Diag}}  \\
		&\le \left\langle \frac{f_{2k+3}(\mathcal N)}{c_{2k+2}} \frac{3}{2\pi \lvert\mathfrak w \rvert}\left(1+\frac{1}{2(k+1)^{1+\delta}}\right) (-\mathcal L_0^{\mathfrak w})\mathrm L^N_{2k+3}\left(\lambda+(-\mathcal L_0),z_{2k+3}(\mathcal N)\right)\varphi, \varphi \right\rangle \\
		& = \left\langle \frac{3\left(1+\frac{1}{2(k+1)^{1+\delta}}\right)}{2\pi \lvert\mathfrak w \rvert c_{2k+2}} (-\mathcal L_0^{\mathfrak w})\mathcal S_{2k+3}^N\varphi, \varphi \right\rangle \, .
	\end{equs}

	Let us now estimate the off-diagonal terms. We want to apply Lemma~\ref{le:control_offdiag}, whose hypotheses require estimates on two sums involving the Fourier multiplier $\zeta^{2k+2}$, one for the off-diagonal terms of type 1 and one for the off-diagonal terms of type 2.

	The sum that need to be estimated for the off-diagonal terms of type 1 is
	\begin{equ} \label{eq:offdiag3}
		\lvert k_2\rvert\sum_{\ell+m=k_1}\frac{1}{{\lvert m \rvert}}\frac{\mathbb J^N_{\ell,m} }{\lambda + \tilde\Gamma +c_{2k+2}\left[\frac{\tilde\Gamma^{\mathfrak w}}{f_{2k+3}}\left(\mathrm L^N_{2k+2}(\lambda + \tilde \Gamma,z_{2k+3})-f_{2k+3}\right)-\frac{\tilde\Gamma}{(n+1+k)^{1+\delta}}\right]} \, ,
	\end{equ}
	where, again, we first used property~(\ref{eq:ex_iter_chaos}) and then suppressed the variable $n$.  As was the case for the diagonal terms, we will first manipulate this expression a bit and then apply to it a lemma proved in the appendix.
	By dropping from the denominator of~(\ref{eq:offdiag3}) the Laplace variable $\lambda$ and all terms involving the Fourier modes $k_{\{1:n\}\setminus\{2\}}$, by lower bounding $\mathrm L^N_{2k+2}$ by $0$ and by applying Cauchy-Schwarz~(\ref{eq:tGamma_comp}), we obtain the upper bound
	\begin{equ}
		\frac{4 \lvert k_2\rvert}{\left(1-\lvert \mathfrak w\rvert^2 c_{2k+2}-\frac{c_{2k+2}}{(n+1+k)^{\delta}}\right)}\sum_{\ell+m=k_1}\frac{1}{{\lvert m \rvert}}\frac{\mathbb J^N_{\ell,m} }{\left(\lvert m \rvert^2 + \lvert k_2 \rvert^2\right)} \lesssim \frac{1}{\left(1-\lvert \mathfrak w\rvert^2 c_{2k+2}-\frac{c_{2k+2}}{(n+1+k)^{\delta}}\right)} \, ,
	\end{equ}
	where we estimated the sum by the corresponding integral and applied Lemma~\ref{le:arctan_parameter} with $\beta = \lvert k_2 \rvert^2$ and $\gamma=1$.

	Regarding the off-diagonal terms of type $2$, instead, we need to estimate the sum
	\begin{equ}\label{eq:offdiag4}
		\sum_{\ell+m=k_1}\frac{1}{{\lvert \ell \rvert \lvert m \rvert}}\frac{\mathbb J^N_{\ell,m} }{\lambda + \tilde\Gamma +c_{2k+2}\left[\frac{\tilde\Gamma^{\mathfrak w}}{f_{2k+3}}\left(\mathrm L^N_{2k+2}(\lambda + \tilde \Gamma,z_{2k+3})-f_{2k+3}\right)-\frac{\tilde\Gamma}{(n+1+k)^{1+\delta}}\right]}\, .
	\end{equ}
	We now observe that the condition $\ell+m = k_1$ implies that at least one between $\ell$ and $m$ has norm larger than the one of $\tfrac12 k_1$. Using this to replace $\lvert \ell\rvert^2$ by $\lvert \tfrac12 k_1 \rvert^2$ and by applying arguments analogous to the ones just used in the estimate of~(\ref{eq:offdiag3}), we obtain that~(\ref{eq:offdiag4}) is upper bounded by
	\begin{equs}
		&\frac{16}{\left(1-\lvert \mathfrak w\rvert^2 c_{2k+2}-\frac{c_{2k+2}}{(n+1+k)^{\delta}}\right)}\sum_{\ell+m=k_1}\frac{1}{{\lvert k_1 \rvert\lvert m \rvert}}\frac{\mathbb J^N_{\ell,m} }{\left(\lvert m \rvert^2 + \lvert k_{1:n} \rvert^2\right)} \\
		& \lesssim \frac{1}{\lvert k_1 \rvert \sqrt{\lvert k_{1:n}\rvert^2}}\frac{\pi^2}{\left(1-\lvert \mathfrak w\rvert^2 c_{2k+2}-\frac{c_{2k+2}}{(n+1+k)^{\delta}}\right)} \, ,
	\end{equs}
	where in the last inequality we estimated the sum by the corresponding integral and applied Lemma~\ref{le:arctan_parameter} with $\beta = \lvert k_{1:n} \rvert^2$ and $\gamma=1$.

	Using Lemma~\ref{le:control_offdiag} with the estimates above as hypotheses, we conclude that:
	\begin{equs}
		&\left\lvert\left\langle \left(\mathcal A_+^N\right)^*\mathcal Z_{2k+2}\mathcal A_+^N \varphi, \varphi \right\rangle_{\mathrm{Off}_1}\right\rvert + \left\lvert\left\langle \left(\mathcal A_+^N\right)^*\mathcal Z_{2k+2}\mathcal A_+^N \varphi, \varphi \right\rangle_{\mathrm{Off}_2}\right\rvert \\
		& \le C_{\mathrm{off}}\left\langle \left(\frac{2\lvert \mathfrak w \rvert^2 (\mathcal N-1)}{1-\lvert \mathfrak w\rvert^2 c_{2k+2}-\frac{c_{2k+2}}{(\mathcal N+1+k)^{\delta}}} + \frac{\lvert \mathfrak w \rvert^2  (\mathcal N-1) (\mathcal N-2)}{1-\lvert \mathfrak w\rvert^2 c_{2k+2}-\frac{c_{2k+2}}{(\mathcal N+1+k)^{\delta}}}\right)(-\mathcal L_0) \varphi, \varphi \right\rangle \\
		& \le C_{\mathrm{off}}\left\langle \frac{\lvert \mathfrak w \rvert^2 \mathcal N^{\,2}}{1-\lvert \mathfrak w\rvert^2 c_{2k+2}-\frac{c_{2k+2}}{(\mathcal N+1+k)^{\delta}}}(-\mathcal L_0) \varphi, \varphi \right\rangle \, ,
	\end{equs}
	where $C_{\mathrm{off}}$ is an absolute constant independent of all variables at play (see~(\ref{eq:variables}) for more details).

	With this we have come to the last part of the proof of the iterative upper bounds. Recall that we were trying to estimate $\langle \left(\mathcal A_+^N\right)^*\mathcal Z_{2k+2}\mathcal A_+^N\varphi,\varphi\rangle$ from above, for $\varphi \in \Gamma L^2_n$. In order to do so, we split it into diagonal and off-diagonal terms and we bounded each of them separately. We now put those estimates together:
	\begin{equs}
		& \left\langle \left(\mathcal A_+^N\right)^*\mathcal Z_{2k+2}\mathcal A_+^N \varphi,\varphi\right\rangle  \\
		& \le \left\langle \left(\mathcal A_+^N\right)^*\mathcal Z_{2k+2}\mathcal A_+^N \varphi,\varphi \right\rangle_{\mathrm{Diag}} + \left \lvert\left\langle \left(\mathcal A_+^N\right)^*\mathcal Z_{2k+2}\mathcal A_+^N \varphi,\varphi\right\rangle_{\mathrm{off}_1}\right\rvert+\left\lvert\left\langle \left(\mathcal A_+^N\right)^*\mathcal Z_{2k+2}\mathcal A_+^N \varphi,\varphi \right\rangle_{\mathrm{off}_2}\right\rvert \\
		& \le \left\langle \left(\frac{3\left(1+\frac{1}{2(k+1)^{1+\delta}}\right)}{2\pi \lvert \mathfrak w \rvert c_{2k+2}} \, \mathcal S_{2k+3}^N +\frac{C_{\mathrm{off}}\lvert \mathfrak w \rvert^2 \mathcal N^{\,2}}{1-\lvert \mathfrak w\rvert^2 c_{2k+2}-\frac{c_{2k+2}}{(\mathcal N+1+k)^{\delta}}}(-\mathcal L_0)\right) \varphi,\varphi \right\rangle  \\
		& \le \left\langle\frac{3\left(1+\frac{1}{2(k+1)^{1+\delta}}\right)}{2\pi \lvert \mathfrak w \rvert c_{2k+2}} \left(\mathcal S_{2k+3}^N +\frac{\pi C_{\mathrm{off}} \lvert \mathfrak w \rvert^3 c_{2k+2}\,\mathcal N^{\,2}}{1-\lvert \mathfrak w\rvert^2 c_{2k+2}-\frac{c_{2k+2}}{(\mathcal N+1+k)^{\delta}}}(-\mathcal L_0)\right) \varphi, \varphi \right\rangle \, ,\label{eq:near_end_UB} \\
	\end{equs}
	where we first used the triangular inequality, then the bounds established in the proof and finally factored out the coefficient of $\mathcal S_{2k+3}^N$, together with the estimate $\left(1+\frac{1}{2(k+1)^{1+\delta}}\right)^{-1} < 1$.

	At this point, the proof is almost complete. We just need a few more estimates on the coefficient of $(-\mathcal L_0)$. First we  multiply and divide it by $f_{2k+3}$ and then we use the estimate
	\begin{equ}
		\frac{\pi C_{\mathrm{off}}\lvert \mathfrak w \rvert^3 c_{2k+2}\,n^2}{\left(1-\lvert \mathfrak w\rvert^2 c_{2k+2}-\frac{c_{2k+2}}{(n+1+k)^{\delta}}\right)f_{2k+3}} \le \frac{C_{\mathrm{off}} \lvert \mathfrak w \rvert^3}{3 K^{2/3}}\frac{n^2}{(n+k)^{3+\delta}} \le 1 \, ,
	\end{equ}
	which holds for $K$ large enough.
	This tells us that expression~(\ref{eq:near_end_UB}) is upper bounded by
	\begin{equ}
		\left\langle\frac{3\left(1+\frac{1}{2(k+1)^{1+\delta}}\right)}{2\pi \lvert \mathfrak w \rvert c_{2k+2}} \left(\mathcal S_{2k+3}^N +f_{2k+3}(\mathcal N)(-\mathcal L_0)\right) \varphi, \varphi \right\rangle \, .
	\end{equ}
	Since the fraction in the above expression is exactly the definition of $c_{2k+3}$ given in the statement of Theorem~\ref{th:iterative_bounds}, the proof is complete.
\end{proof}

\section{Proof of the main theorem}\label{se:mainproof}
We will now use Theorem~\ref{th:iterative_bounds} to prove Theorem~\ref{th:main}.

\begin{proof}[Proof of Theorem~\ref{th:main}]
	The strategy of the proof is the following. First of all we apply Proposition~\ref{pr:Laplace_bulk}, so to reduce our problem to the one of finding estimates from above and from below on the quantity $\left\langle \mathfrak n^N, \left(\lambda N^2 - \mathcal L^N)^{-1}\right) \mathfrak n^N \right\rangle$. This is done by using the upper and lower bounds provided by Lemma~\ref{le:yau}, which we first simplify thanks to Lemma~\ref{le:innerproduct} and then further estimate with Theorem~\ref{th:iterative_bounds}.

	We start with the upper bound. We have
	\begin{equs}
		& \left\langle \mathfrak n^N, \left(\lambda N^2 - \mathcal L^N\right)^{-1} \mathfrak n^N \right\rangle \\
		& \le
		\left\langle \mathfrak n^N , \left((\lambda N^2 - \mathcal L_0)+\mathcal H^N_{2k+2}\right)^{-1} \mathfrak n^N \right\rangle \\
		& \le \left\langle \mathfrak n^N, \left( (\lambda N^2 - \mathcal L_0) +  c_{2k+2} \left((-\mathcal L_0^{\mathfrak w}) \mathcal S_{2k+2} - \frac{1}{(\mathcal N +k)^{1+\delta}}(-\mathcal L_0)\right)  \right)^{-1}\mathfrak n^N\right\rangle \, ,
	\end{equs}
	where the first inequality follows from Lemma~\ref{le:yau} and Lemma~\ref{le:innerproduct} and the second one from  estimate~(\ref{eq:iter_LB}) on $\mathcal H_{2k+2}$ given by Theorem~\ref{th:iterative_bounds}.
	Recalling the Fourier expression of $\mathfrak n^N$ given in~(\ref{eq:n_Fourier}), the above scalar product is exactly twice the sum~(\ref{eq:horror3}), written for $n=2$, $k_{2:n}=0$, $k_1 =0$  and Laplace variable $\lambda N^2$.
	Following exactly the same steps performed there (compare with~(\ref{eq:less_horror}) and include the factor that multiplies the sum to which Lemma~\ref{le:int_est} is applied), we obtain the upper bound 
	\begin{equ}
		\frac{f_{2k+3}(2)}{c_{2k+2}}\frac{3\pi}{\lvert\mathfrak w \rvert}\left(1+\frac{1}{2(k+1)^{1+\delta}}\right) \mathrm L^N_{2k+3}\left(\lambda N^2,z_{2k+3}(2)\right) \, . 
	\end{equ}
	By recalling the definition of $\mathrm L^N_k$, $f_{2k+3}$ and $z_{2k+3}$ given by~(\ref{eq:log}) and~(\ref{eq:z_and_f}), we can further estimate it by
	\begin{equs}
		&\frac{9 K^{\frac{2}{3}}(2k+5)^{3+\delta}}{c_{2k+2}}\frac{\pi}{\lvert\mathfrak w \rvert}\left(1+\frac{1}{2(k+1)^{1+\delta}}\right) \mathrm L_{2k+3}\left(\lambda ,K(2k+5)^{\frac{9}{2}+\frac{3}{2}\delta}\right) \\
		&\lesssim  C(\lvert\mathfrak w\rvert)k^{3+\delta}\left(\left({\log}\left(1+\lambda^{-1}\right)\right)^{\theta_{2k+3}}+ k^{\frac{9}{2}+\frac{3}{2}\delta}\right) \\
		 &= C(\lvert\mathfrak w\rvert) \left[k^{3+\delta}\left({\log}\left(1+\lambda^{-1}\right)\right)^{\theta_{2k+3}-\frac{2}{3}}+k^{\frac{15}{2}+\frac{5}{2}\delta}\left({\log}\left(1+\lambda^{-1}\right)\right)^{-\frac{2}{3}}\right]\times\\
  &\quad\times\left({\log}\left(1+\lambda^{-1}\right)\right)^{\frac{2}{3}}\,,\label{eq:UB_for_bulk}
	\end{equs}
	where in the inequality we used both that $c_{2k+2}$ is bounded away from $0$ and infinity uniformly in $k$. Expression~(\ref{eq:UB_for_bulk}) provides us with a valid upper bound for each value of $k$, with the best one being the one that minimizes the factor in front of $\left(\log\left(1+\lambda^{-1}\right)\right)^{\frac{2}{3}}$. We choose
	\begin{equ} \label{eq:k_of_mu}
		k = k(\lambda) = \left\lfloor (\log4)^{-1}\log\log\log\left(1+\lambda^{-1}\right) \right\rfloor \, ,
	\end{equ}
	which is greater than or equal to $0$ if $\lambda$ is such that $1+\lambda^{-1} \ge e^e$.
	Recalling the close formula for $\theta_k$ given in~(\ref{eq:explicit_alpha}), this gives us the estimates
	\begin{equs}
		\theta_{2k(\lambda)+3} &=  \frac{2}{3} + \frac{1}{3}\left(\frac{1}{4}\right)^{\left\lfloor (\log4)^{-1}\log\log\log\left(1+\lambda^{-1}\right) \right\rfloor} \le \frac{2}{3} + \frac{1}{3\log\log\left(1+\lambda^{-1}\right)} \, , \\
		k^{3+\delta} &\le \left(\log\log\log\left(1+\lambda^{-1}\right)\right)^{3+\delta} \, ,\\
		k^{\frac{15}{2}+\frac{5}{2}\delta}\left({\log}\left(1+\lambda^{-1}\right)\right)^{-\frac{2}{3}} & \le \left(\log\log\log\left(1+\lambda^{-1}\right)\right)^{\frac{15}{2}+\frac{5}{2}\delta}\left({\log}\left(1+\lambda^{-1}\right)\right)^{-\frac{2}{3}} \lesssim 1 \, ,\\
		\left(\mathrm{log}\left(1+\lambda^{-1}\right)\right)^{\theta_{2k+3}-\frac{2}{3}} & \le \left(\mathrm{log}\left(1+\lambda^{-1}\right)\right)^{\frac{1}{3\log\log\left(1+\lambda^{-1}\right)}} = \sqrt[3]{e} \, ,
	\end{equs}
	where the second inequality in the second to last line is justified by the fact that the left-hand side goes to $0$ as $\lambda \to 0$.

	Summing up, upper bounding expression~(\ref{eq:UB_for_bulk}) by the estimates above and recalling the expression of $\mathcal D^N(\lambda)$ derived in Proposition~\ref{pr:Laplace_bulk}, we obtain
	\begin{equs}
		\mathcal D^N(\lambda) &= \frac{1}{\lambda^2}+\frac{1}{\lambda^2}\lvert\mathfrak w\rvert^2 \left\langle \mathfrak{n}^N,\left(\lambda N^2-\mathcal L^N\right)^{-1}\mathfrak{n}^N\right\rangle  \\
		& \lesssim \frac{C(\lvert\mathfrak w\rvert)}{\lambda^2} \left(\log\log\log\left(1+\lambda^{-1}\right)\right)^{3+\delta}\left({\log}\left(1+\lambda^{-1}\right)\right)^{\frac{2}{3}} \, .
	\end{equs}
    Since the above inequality holds for every $N \in \mathbb N$, taking $\limsup_{N \to \infty}$ on both sides and observing that $\log\left(1+\lambda^{-1}\right) \sim_{\lambda \to 0} \lvert\log(\lambda)\rvert$ proves the upper bound of Theorem~\ref{th:main}.

	We now proceed to the lower bound. We have
	\begin{equs}
		&\left\langle \mathfrak n^N, \left(\lambda N^2 - \mathcal L^N\right)^{-1} \mathfrak n^N \right\rangle \\
        &\ge \left\langle \mathfrak n^N, \left((\lambda N^2 - \mathcal L_0)+\mathcal H^N_{2k+1}\right)^{-1}\mathfrak n^N \right\rangle \\
		& \ge \left\langle \mathfrak n^N, \left((\lambda N^2 - \mathcal L_0)+c_{2k+1} \left((-\mathcal L_0^{\mathfrak w})  \mathcal S_{2k+1}+ f_{2k+1}(\mathcal N)(-\mathcal L_0)\right) \right)^{-1}\mathfrak n^N \right\rangle \, ,
	\end{equs}
	where the first inequality follows again from Lemma~\ref{le:yau} and Lemma~\ref{le:innerproduct} and the second one from  estimate~(\ref{eq:iter_UB}) on $\mathcal H_{2k+1}$ given by Theorem~\ref{th:iterative_bounds}.
	Again, recalling the Fourier expression of $\mathfrak n^N$ given in~(\ref{eq:n_Fourier}), the above scalar product is exactly twice the sum~(\ref{eq:horror1}), written for $n=2$, $k_{2:n}=0$, $k_1 =0$ and Laplace variable $\lambda N^2$. Following exactly the same steps performed there (compare with~(\ref{eq:final})), we obtain the lower bound 
	\begin{equ}
		\frac{3\pi}{\lvert \mathfrak w \rvert}\frac{1}{c_{2k+1}}\left(1-\frac{1}{2k^{1+\delta}}\right)\frac{1}{f_{2k+2}(2)}\left[\mathrm L^N_{2k+2}\left(\lambda N^2, z_{2k+2}(2)\right) - f_{2k+2}(2)\right] \, .
	\end{equ}
	Recalling again the definitions of $\mathrm L^N_k$, $f_{2k+3}$ and $z_{2k+3}$ given by~(\ref{eq:log}) and~(\ref{eq:z_and_f}), and using the fact that $c_{2k+1}$ is bounded away from $0$ and infinity uniformly in $k$, the above is further lower bounded by
	\begin{equs}
		&\frac{\pi}{\lvert \mathfrak w \rvert c_{2k+1}K^{\frac23}(2k+4)^{3+\delta}}\left[\mathrm L_{2k+2}\left(\lambda, K(2k+4)^{\frac{9}{2}+\frac{3}{2}\delta}\right) - 3K^{\frac23}(2k+4)^{3+\delta}\right] \\
		& \gtrsim
		\frac{C(\lvert\mathfrak w\rvert)}{k^{3+\delta}}\left(\mathrm{log}\left(1+\lambda^{-1}\right) + k^{\left(\frac{9}{2}+\frac{3}{2}\delta\right)}\right)^{\theta_{2k+2}} \\
		& \ge \frac{C(\lvert\mathfrak w\rvert)}{k^{3+\delta}}\frac{1}{\left(\left({\log}\left(1+\lambda^{-1}\right)\right)^{\frac{2}{3}-\theta_{2k+2}} + k^{\left(\frac{9}{2}+\frac{3}{2}\delta\right)\left(\frac{2}{3}-\theta_{2k+2}\right)}\right)}\left({\log}\left(1+\lambda^{-1}\right)\right)^{\frac{2}{3}} \, ,\label{eq:LB_for_bulk}
	\end{equs}
	where we have used that fact that $L_{2k+2}$ goes to infinity as $\lambda\to0$ to absorb the $-1$ in the multiplicative constant.  As before, this gives a valid lower bound for each choice of $k$, this time with the best one being the one that maximizes the factor in front of $\mathrm{log}^{\frac{2}{3}}\left(1+\lambda^{-1}\right)$. We use the same choice made for the upper bound, namely~(\ref{eq:k_of_mu}). Recalling the close formula for $\theta_k$ given in~(\ref{eq:explicit_alpha}), this gives us the estimates
	\begin{equs}
        \theta_{2k(\lambda) +2} &= \frac{2}{3} - \frac{2}{3}\left(\frac{1}{4}\right)^{\left\lfloor (\log4)^{-1}\log\log\log\left(1+\lambda^{-1}\right) \right\rfloor} \ge \frac{2}{3} - \frac{1}{3\log\log\left(1+\lambda^{-1}\right)} \, , \\
		\left({\log}\left(1+\lambda^{-1}\right)\right)^{(\frac{2}{3}-\theta_{2k+2})} &\le \left(\mathrm{log}\left(1+\lambda^{-1}\right) \right)^{\frac{1}{3\log\log\left(1+\lambda^{-1}\right)}} = \sqrt[3]{e} \, ,\\
		k^{\left(\frac{9}{2}+\frac{3}{2}\delta\right)\left(\frac{2}{3}-\theta_{2k+2}\right)} &\le \left(\log\log\log\left(1+\lambda^{-1}\right)\right)^{\frac{9+3\delta}{6 \log\log\left(1+\lambda^{-1}\right)} } \le \sqrt[6]{e} \, ,\\
		k^{3+\delta} &\le \left(\log\log\log\left(1+\lambda^{-1}\right)\right)^{3+\delta} \, ,
	\end{equs}
 	with which we can upper bound all terms that appear in the denominator of~(\ref{eq:LB_for_bulk}).

	Summing up, lower bounding expression~(\ref{eq:LB_for_bulk}) by the estimates above and recalling the expression of $\mathcal D^N(\lambda)$ derived in Proposition~\ref{pr:Laplace_bulk}, we obtain
	\begin{equs}
		\mathcal D^N(\lambda) &= \frac{1}{\lambda^2}+\frac{1}{\lambda^2}\lvert\mathfrak w\rvert^2 \left\langle \mathfrak{n}^N,\left(\lambda N^2-\mathcal L^N\right)^{-1}\mathfrak{n}^N \right\rangle  \\
		& \gtrsim \frac{C(\lvert\mathfrak w\rvert)}{\lambda^2} \left(\log\log\log\left(1+\lambda^{-1}\right)\right)^{-3-\delta}\left({\log}\left(1+\lambda^{-1}\right)\right)^{\frac{2}{3}} \, .
	\end{equs}
    Since the above inequality holds for every $N \in \mathbb N$, taking $\liminf_{N \to \infty}$ on both sides and observing that $\log\left(1+\lambda^{-1}\right) \sim_{\lambda \to 0} \lvert\log(\lambda)\rvert$ proves the lower bound of Theorem~\ref{th:main}.
\end{proof}

\appendix
\section{Replacement Lemmas} \label{ap:replacement_lemmas}
The present Appendix is devoted to estimating sums corresponding to the hypothesis of Lemma~\ref{le:control_diag} in the context of the proof of Theorem~\ref{th:iterative_bounds}.

We start by stating some useful identities and setting up some notation.

\begin{lemma} \label{le:arctan_parameter}
	For every $\beta>0$ and $\gamma>0$
	\begin{equ}
		\int_0^{+\infty} \! \! \!\frac{1}{\beta+\gamma r^2} \dd r = \frac{\pi}{2\sqrt{\beta\gamma}}\, , \qquad \int_0^{\pi}\! \frac{1}{\beta+\gamma (\cos\theta)^2}\dd \theta = \frac{\pi}{\sqrt{\beta(\beta+\gamma)}} \, .
	\end{equ}
\end{lemma}
The proof is omitted, since it is a change of variables of standard integrals.

Since this Appendix concerns expressions involving a large number of variables, and it is of technical nature anyway, let us list them all here once and for all, together with their range:
\begin{equ}\label{eq:variables}
	\lambda \in (0,+\infty)\, , \quad z \in (1,+\infty)\, , \quad \ell,m,k_1,\dots,k_n\in \mathbb Z^2_0\, , \quad N\in \mathbb N\, , \quad k\in \mathbb N, \, \, k\ge 2\, , \quad \mathfrak w \in \mathbb R^2 \, .
\end{equ}
In particular, the constant that we omit when using the notation $\lesssim$, which was introduced in subsection~\ref{su:notation}, is independent of all variables listed above. 

Let us start by introducing some additional notation:
\begin{equs}
	\alpha &\eqdef \alpha(\lambda,k_{1:n})\eqdef \lambda +\lvert k_{1:n}\rvert^2 \, , &\qquad \alpha_N &\eqdef \frac{\alpha}{N^2} \, ,\\
	\Gamma &\eqdef \Gamma(\ell,m,k_{2:n}) \eqdef \vert\ell\rvert^2 +\tfrac12\left(\lvert k_{1:n} \rvert^2\right) \, ,  & \qquad \Gamma^{\mathfrak w}  &\eqdef \Gamma^{\mathfrak w}(\ell,m,k_{2:n})\eqdef (\mathfrak w \cdot\ell)^2+\tfrac12\left((\mathfrak w\cdot k)^2_{1:n}\right) \, .
\end{equs}
Recall also the definitions of $\tilde\Gamma$ and $\tilde\Gamma^{\mathfrak w}$ given in~(\ref{eq:tildeGamma}). It is useful to observe that
\begin{equ}\label{eq:to_tilde_or_not_to_tilde}
	\tilde \Gamma = \Gamma - \ell\cdot k_1 \, , \qquad  \tilde\Gamma^{\mathfrak w} = \Gamma^{\mathfrak w} - (\mathfrak w \cdot k_1)(\mathfrak w \cdot \ell) \, , \qquad \frac{1}{2} \le \Gamma \lesssim \tilde\Gamma \lesssim \Gamma \, .
\end{equ}
Finally, we take note of the following derivatives, that will be needed later on:
\begin{equ}
	\partial_x \mathrm L(x,z) = -\frac{1}{x(x+1)} \, ,
	\qquad
	\partial_x \mathrm L^N_k(x,z) = -\theta_k \left(\mathrm L^N_k(x,z)\right)^{\theta_k-1}\frac{N^2}{x(x+N^2)} \, .
\end{equ}

Our goal is to study sum~(\ref{eq:moving_parts_S}) below, which is the one that appears in the estimates of the diagonal terms in Theorem~\ref{th:iterative_bounds}. First, in Lemma~\ref{le:replacement}, we replace the sum with an integral. Then, in Lemma~\ref{le:int_est}, we replace this integral with another one, which admits an explicit primitive. The first replacement comes at the price of an additive constant, the second one at the price of a lower-order term. 

We set
\begin{equs} 
	\tilde{\mathrm S}&\eqdef \tilde{\mathrm S}(\lambda,N,k,k_{1:n},\mathfrak w) \eqdef \sum_{\ell+m=k_1}\frac{\mathbb J^N_{\ell,m} }{\lambda + \tilde\Gamma +\tilde\Gamma^{\mathfrak w}\mathrm L^N_{k}\left(\lambda + \tilde \Gamma,z\right)} \, , \label{eq:moving_parts_S} \\ 
	\mathrm I &\eqdef \mathrm I(\lambda,N,k,k_{1:n},\mathfrak w) \\
    &\eqdef \int_0^{\pi} \int_0^1 \frac{1}{(r+\alpha_N)(r+\alpha_N+1)\left(1+ \lvert\mathfrak w\rvert^2 (\cos\theta)^2 \mathrm L^N_k\left(N^2(r+\alpha_N),z\right)\right)} \dd r \dd \theta \, . \label{eq:moving_parts}
\end{equs}
The definition of $\mathrm I$ is motivated by the change of variables~(\ref{eq:change_var}) below.
\begin{lemma}[From sum to integral] \label{le:replacement}
	There exists a constant $C_{\mathrm{Diag}}>0$ such that 
	\begin{equ}
		\lvert \tilde{\mathrm S} - \mathrm I \rvert \le C_{\mathrm{Diag}}
	\end{equ}
	for all in $\lambda \in (0,+\infty)$, $z \in (1,+\infty)$, $k_{1:n}\in  (\mathbb Z^2_0)^n$, $N\in \mathbb N$, $k\in \mathbb N$ such that $k\ge 2$ and $\mathfrak w \in \mathbb R^2$.
\end{lemma}
\begin{remark}\label{re:w_dep}
    The proof of Theorem~\ref{th:iterative_bounds} would work even if $C_{\mathrm{Diag}}$ depended on $\lvert \mathfrak w \rvert$ (up to changing the constant $K$). In the following proof, however, showing that $C_{\mathrm{Diag}}$ does not depend on $\mathfrak w$ does not come with any significant additional difficulties.
\end{remark}
\begin{proof}
	The proof proceeds through a number of steps, each of them consisting in slightly modifying the expression of $\tilde{\mathrm S}$, at the price of an additive constant, so that it becomes closer to the one of $\tilde{\mathrm I}$. More precisely, set $\mathrm{S}_0 \eqdef \tilde{\mathrm S}$ and $\mathrm S_6 \eqdef \mathrm I$. Then for all $i\in\{1,\dots,6\}$, step $i$ consists in showing $\lvert \mathrm S_{i-1} - \mathrm S_{i}\rvert \le C_i$, where the $\mathrm S_i$'s for  $i\in \{1,2,3,4,5\}$ will be defined in the proof below and $C_i$'s are some absolute constants independent of all variables at play. Compared to \cite[Appendix A]{cannizzaro2023gaussian}, we are faced with some additional technical difficulties, coming from the fact that our equation is in the strong coupling regime. 
	
    \medskip
	\noindent {\bf Step 1} We define $\mathrm S_1$ by replacing the condition $\lvert k_1 \rvert\le N$, contained in $\mathbb J_{\ell,m}^N$, by $\lvert k_1 \rvert \le N/2$. In doing so, we lose all summands of $\tilde{\mathrm S}$ corresponding to $\ell + m = k_1 \in [N/2,N]$. Without loss of generality, suppose $\lvert \ell \rvert \ge N/4$. The computation
	\begin{equ}
		\lvert \tilde{\mathrm S} - \mathrm S_1 \rvert \le 2 \sum_{\substack{\ell+m=k_1 \\ \frac{N}{2} \le \lvert k_1 \rvert \le N}}\frac{\mathbf 1_{\{N/4 \le \lvert \ell \rvert \le N\}}\, \mathbf 1_{\{\lvert k_1-\ell \rvert \le N\}}}{\lambda + \tilde\Gamma +\tilde\Gamma^{\mathfrak w}\mathrm L^N_{k}\left(\lambda + \tilde \Gamma,z\right)} \lesssim \sum_{\substack{\ell \in \mathbb Z^2_0 \\\frac{N}{4} \le \lvert \ell \rvert \le N}} \frac{1}{\lvert \ell \rvert^2} \le C_1
	\end{equ}
	completes step 1.
 
    \medskip
	\noindent {\bf Step 2} We define $\mathrm S_2$ by 
    \begin{equ}
        \mathrm S_2 \eqdef \sum_{\ell+m=k_1}\frac{\mathbb J^N_{\ell,m} \mathbf 1_{ \{ 1 \le \lvert k_1 \rvert \le N/2 \} } }{\lambda + \Gamma +\Gamma^{\mathfrak w}\mathrm L^N_{k}\left(\lambda + \Gamma,z\right)} \, .
    \end{equ}
    With respect to $\mathrm S_1$, we replaced $\tilde \Gamma$ and $\tilde \Gamma^{\mathfrak w}$ by $\Gamma$ and $\Gamma^{\mathfrak w}$ respectively.
	Using relationship~(\ref{eq:to_tilde_or_not_to_tilde}) on $\tilde \Gamma^{\mathfrak w}$ and the triangular inequality, we get that
	\begin{equ}\label{eq:horrorA1}
		\lvert \mathrm S_1 - \mathrm S_2 \rvert \le \sum_{\substack{\ell +m = k_1\\ 1 \le \lvert \ell \rvert, \lvert m \rvert \le N, \, \lvert k_1 \rvert \le \frac{N}{2} }} \frac{A+B+D}{\left[\lambda + \tilde\Gamma +\tilde\Gamma^{\mathfrak w}\mathrm L^N_{k}\left(\lambda + \tilde \Gamma,z\right)\right]\left[\lambda + \Gamma +\Gamma^{\mathfrak w}\mathrm L^N_{k}\left(\lambda +  \Gamma,z\right)\right]} \, ,
	\end{equ}
	where
	\begin{equs}
		& A \eqdef \left\lvert \Gamma - \tilde\Gamma\right\rvert \, ,
		\quad
		B \eqdef \left\lvert \Gamma^{\mathfrak w} \left(\mathrm L^N_k\left(\lambda+\Gamma,z\right)-\mathrm L^N_k\left(\lambda+\tilde\Gamma,z\right)\right) \right\rvert \, , \\
		& D \eqdef \left\lvert (\mathfrak w \cdot k_1)(\mathfrak w \cdot \ell) \mathrm L^N_k\left(\lambda+\tilde\Gamma,z\right)\right\rvert \, .
	\end{equs}
	We estimate~(\ref{eq:horrorA1}) by considering the terms with $A$, $B$ and $D$ separately.

	First, by dropping some terms from the denominator of~(\ref{eq:horrorA1}) (they are all positive) and using~(\ref{eq:to_tilde_or_not_to_tilde}) for $\tilde\Gamma$, we obtain the following upper bound for the $A$ term:
	\begin{equ}\label{eq:one_of_many}
		\sum_{\substack{\ell +m = k_1\\ 1 \le \lvert \ell \rvert, \lvert m \rvert \le N, \, \lvert k_1 \rvert \le \frac{N}{2}}} \frac{\lvert \ell\cdot k_1 \rvert}{\tilde \Gamma \, \Gamma} \lesssim \sum_{\lvert \ell \rvert \le N} \frac{\lvert \ell \cdot k_1 \rvert}{(\lvert \ell \rvert^2+\lvert k_1 \rvert^2)^2 } \lesssim 1\, ,
	\end{equ}
	where one can check that the constant on the right-hand side of~(\ref{eq:one_of_many}) above is independent of $k_1$  by splitting the sum into the two regions $\lvert \ell \rvert \ge \lvert k_1 \rvert$ and $\lvert \ell \rvert < \lvert k_1 \rvert$.

	Then, by the mean value theorem applied to the function $\mathrm L^N_k$ and the interval $[a,b]$, where $a\eqdef (\lambda+\tilde\Gamma)\wedge (\lambda+\Gamma)$ and $b\eqdef (\lambda+\tilde\Gamma) \vee (\lambda +\Gamma)$, we obtain
	\begin{equ}
		B \lesssim \Gamma^{\mathfrak w} \sup_{y \in [a,b]}\left\rvert \frac{N^2}{y(y+N^2)} \right\rvert  \left\lvert \tilde \Gamma - \Gamma \right\rvert \lesssim \Gamma^{\mathfrak w} \frac{1}{\Gamma}\lvert \ell \cdot k_1\rvert \, ,
	\end{equ}
	where we estimated the derivative of $\mathrm L^N_k$ by first using $\theta_k-1 < 0$ and $\mathrm L^N_k \ge 1$ and then lower bounding $y+N^2 \ge N^2$, so that $N^2$ cancels. Finally, we used~(\ref{eq:to_tilde_or_not_to_tilde}) again. Thus, by dropping $\lambda$ and $\Gamma$ from the second factor in the denominator of~(\ref{eq:horrorA1}) and lower bounding $\mathrm L^N_k$ by $1$, the term of sum~(\ref{eq:horrorA1}) corresponding to $B$ is upper bounded by
	\begin{equ}
		\sum_{\lvert \ell \rvert \le N} \frac{\Gamma^{\mathfrak w} \lvert \ell \cdot k_1 \rvert }{\Gamma\,\tilde \Gamma \, \Gamma^{\mathfrak w}} = \sum_{\lvert \ell \rvert \le N} \frac{ \lvert \ell \cdot k_1 \rvert }{\Gamma\,\tilde \Gamma} \lesssim 1 \, ,
	\end{equ}
	where the last inequalities follows from the same argument used in estimate~(\ref{eq:one_of_many}).

	Finally, we need to bound the sum corresponding to the $D$ term. First of all, observe that if $\mathfrak w \cdot k_1 =0$, then $D$ is identically $0$ and so in the following we can assume $\mathfrak w \cdot k_1 \neq 0$. In particular, this guarantees that for every $\ell$ (and for every $\theta$ when we will write the integral) the denominator in the following expressions does not vanish. We then drop some terms from the denominator, so to obtain an upper bound in which the function $\mathrm L^N_k$ has simplified:
	\begin{equ}
		\sum_{\substack{\ell +m = k_1\\ 1 \le \lvert \ell \rvert, \lvert m \rvert \le N,\, \lvert k_1 \rvert \le \frac{N}{2}}} \! \! \!  \frac{\left\lvert (\mathfrak w \cdot k_1)(\mathfrak w \cdot \ell) \mathrm L^N_k\left(\lambda+\tilde\Gamma,z\right)\right\rvert}{\left[\tilde\Gamma^{\mathfrak w}\mathrm L^N_{k}\left(\lambda + \tilde \Gamma,z\right)\right]\Gamma}
		\lesssim
		\sum_{1\le\lvert \ell \rvert \le N }  \frac{\lvert (\mathfrak w \cdot k_1)(\mathfrak w \cdot \ell) \rvert}{\left[(\mathfrak w \cdot \ell)^2+(\mathfrak w \cdot k_1)^2\right] \lvert \ell \rvert^2   } \, .
	\end{equ}
	Finally, we check that the right-hand side of the above can be upper bounded by a convergent series whose sum, as usual, does not depend on any of the variables at play. We do this by passing to an integral.
    This is justified after excluding $\ell$ such that $\lvert\mathfrak w\cdot \ell\rvert\leq \lvert\mathfrak w\rvert$, which can be treated separately.
    For more details see Step 5, where this is done carefully for the main term.
    We write this integral using polar coordinates:
	\begin{equs}
		\int_0^{2\pi} \int_0^N \frac{\lvert (\mathfrak w \cdot k_1)\rvert \lvert \mathfrak w \rvert \lvert \cos(\theta)\rvert  r^2 }{\left[\lvert\mathfrak w \rvert^2 (\cos\theta)^2 r^2  + (\mathfrak w \cdot k_1)^2 \right]r^2 } \dd r \dd \theta & =  \int_0^{2\pi} \int_0^N \frac{\lvert (\mathfrak w \cdot k_1) \rvert \lvert \mathfrak w \rvert  \lvert \cos(\theta)\rvert}{\lvert\mathfrak w \rvert^2  (\cos\theta)^2  r^2+ (\mathfrak w \cdot k_1)^2  }\dd r \dd \theta \\
		& \leq \frac{\pi}{2} \int_0^{2 \pi }\frac{\lvert (\mathfrak w \cdot k_1) \rvert \lvert \mathfrak w \rvert \lvert \cos(\theta)\rvert }{\sqrt{\lvert\mathfrak w \rvert^2  (\cos\theta)^2 (\mathfrak w \cdot k_1)^2 }} \dd r \dd \theta = \pi^2 \, ,
	\end{equs}
	where we used Lemma~\ref{le:arctan_parameter} to compute the integral in $r$.

	Thus step 2 is completed with $C_2$ equal to the sum of the three constants with which we have estimated the sum corresponding to the $A$, $B$ and $D$ terms.

    \medskip
	\noindent {\bf Step 3} We define $\mathrm S_3$ by 
    \begin{equ}
        \mathrm S_3 \eqdef \sum_{\ell+m=k_1}\frac{\mathbb J^N_{\ell,m} \mathbf 1_{ \{ 1 \le \lvert k_1 \rvert \le N/2 \} } }{\lambda + \Gamma +(\mathfrak w \cdot \ell)^2\mathrm L^N_{k}\left(\lambda + \Gamma,z\right)} \, .
    \end{equ} 
    With respect to $\mathrm S_2$, we have replaced $\Gamma^{\mathfrak w}$ by $(\mathfrak w \cdot \ell)^2$. If $(\mathfrak w \cdot k)^2_{1:n} = 0$, we do not have anything to prove. Otherwise, we estimate 
	\begin{equs}
		\lvert \mathrm S_2 - \mathrm S_3 \rvert & \lesssim \sum_{\substack{\ell +m = k_1\\ 1 \le \lvert \ell \rvert, \lvert m \rvert \le N, \, \lvert k_1 \rvert \le \frac{N}{2} }} \frac{(\mathfrak w \cdot k)^2_{1:n} \, \mathrm L^N_k\left(\lambda+\Gamma,z\right)}{\left[\lambda + \Gamma +\Gamma^{\mathfrak w}\mathrm L^N_{k}\left(\lambda + \Gamma,z\right)\right]\left[\lambda + \Gamma +(\mathfrak w \cdot \ell)^2\mathrm L^N_{k}\left(\lambda +  \Gamma,z\right)\right]} \\
		& \lesssim \sum_{1\le\lvert \ell \rvert \le N } \frac{(\mathfrak w \cdot k)^2_{1:n}}{\left[(\mathfrak w \cdot \ell )^2 +(\mathfrak w \cdot k)^2_{1:n}\right]\left[\lvert \ell \rvert^2 + \lvert k_{1:n}\rvert^2\right]} \\
		& \lesssim \int_0^N \int_0^{2\pi}  \frac{(\mathfrak w \cdot k)^2_{1:n} r}{(\lvert \mathfrak w \rvert^2  r^2 (\cos\theta)^2+ (\mathfrak w \cdot k)^2_{1:n})(r^2+\lvert k_{1:n}\rvert^2)}  \dd \theta \dd r\\
		& \lesssim \int_0^N \frac{(\mathfrak w \cdot k)^2_{1:n}}{r^2+\lvert k_{1:n}\rvert^2} \frac{r}{\sqrt{(\mathfrak w \cdot k)^2_{1:n}((\mathfrak w \cdot k)^2_{1:n}+\lvert \mathfrak w \rvert^2 r^2)}} \dd r  \, ,
	\end{equs}
	where we first dropped $\lambda + \Gamma$ and $(\mathfrak w \cdot \ell)^2 \mathrm L^N_k$ from the first and second factor in the denominator respectively, then simplified $\mathrm L^N_k$ and finally used Lemma~\ref{le:arctan_parameter} to estimate the integral in $\theta$.
    Before moving from the sum to the integral, one once again needs to exclude $\ell$ such that $\lvert\mathfrak w\cdot\ell\rvert\leq\lvert\mathfrak w\rvert$.
    These can again easily be treated separately.
    We now simplify the multiplicative factor $(\mathfrak w \cdot k)^2_{k_{1:n}}$ in the denominator and drop the additive one, so that we can simplify $r$ and obtain the upper bound
	\begin{equs}
		\int_0^N \frac{\lvert \mathfrak w \cdot k \rvert_{1:n}}{r^2+\lvert k_{1:n}\rvert^2} \frac{1}{\lvert \mathfrak w \rvert} \dd r \lesssim \frac{\lvert \mathfrak w \rvert \lvert k_{1:n} \rvert}{\lvert \mathfrak w \rvert} \frac{1}{\sqrt{\lvert k_{1:n}\rvert^2}} = 1 \, ,
	\end{equs}
	where we applied the Cauchy-Schwarz inequality to the numerator and estimated the integral by using Lemma~\ref{le:arctan_parameter}.
	This concludes step 3.

    \medskip
	\noindent {\bf Step 4} We define $\mathrm S_4$ by 
    \begin{equ}
		\mathrm S_4 \eqdef \sum_{\ell+m=k_1}\frac{\mathbf 1_{ \{ 1 \le \lvert \ell \rvert \le N \} } \mathbf 1_{ \{ 1 \le \lvert k_1 \rvert \le N/2 \} }}{\lambda + \Gamma +(\mathfrak w \cdot \ell)^2\mathrm L^N_{k}\left(\lambda + \Gamma,z\right)} \, .
	\end{equ}
    With respect to $\mathrm S_3$, we have removed the constraint $\mathbf 1_{\{\lvert m \rvert \le N\}}$. More precisely, we are adding to $\mathrm S_3$ the terms indexed by the set
	\begin{equ}
		\left\{\ell,m \in \mathbb Z^2_0 \, \colon \ell + m = k_1, \, \lvert k_1 \rvert \le \tfrac{N}{2}, \, \lvert \ell \rvert \le N, \, \lvert m \rvert > N \right\} \, ,
	\end{equ}
	which is contained (thanks to the extra condition on $\lvert k_1 \rvert$ imposed in step 1) in
	\begin{equ}
		\left\{\ell,m \in \mathbb Z^2_0 \, \colon \ell + m = k_1, \, \lvert k_1 \rvert \le \tfrac{N}{2}, \, \lvert \ell \rvert \le N, \, \lvert \ell \rvert \ge \tfrac{N}{2}\right\} \, .
	\end{equ}
	The sum over this last index set can be bounded as done in step 1.

    \medskip
	\noindent {\bf Step 5} We define $\mathrm S_5$ by
    \begin{equ}\label{eq:S5}
        \mathrm S_5 \eqdef \int_{\mathbb R^2} \frac{\mathbf1_{\{\lvert x \rvert \le 1\}}}{\alpha_N+\lvert x \rvert^2 + (\mathfrak w \cdot x)^2 \mathrm L^N_k\left(N^2(\alpha_N+\lvert x \rvert^2), z\right)} \dd x\, . 
    \end{equ}
	We set $Q^N_{\ell}\eqdef \tfrac1N [\ell-\tfrac{1}{2},\ell-\tfrac{1}{2}]^2 \subset \mathbb R^2$ and by multiplying and dividing $\mathrm S_4$ by $\tfrac{1}{N^2}$ we obtain
	\begin{equ}\label{eq:S4Sum}
		\mathrm S_4
        =
        \sum_{\ell+m=k_1}  \mathbf 1_{ \{ 1 \le \lvert \ell \rvert \le N \} } \mathbf 1_{ \{ 1 \le \lvert k_1 \rvert \le N/2 \} } \int_{Q^N_\ell} I_N\left(\tfrac{\ell}{N}\right) \dd x \, ,
	\end{equ}
	where we denoted by $I_N$ the integrand of $\mathrm S_5$ (without the indicator function $\mathbf1_{\{\lvert x \rvert \le 1\}}$).
	To show $\lvert\mathrm S_4 - \mathrm S_5\rvert \le 
        C_5$,
        we write $\mathrm S_5$ as the sum in $\ell$ of the integrals over $Q_\ell^N$.
        Since for large $N$ the summand $I_N(x)$ changes very rapidly when $x$ and $\mathfrak w$ are almost orthogonal, we will treat this case separately.
        Note first that
        \begin{equ}
            \sum_{\ell+m=k_1}\frac{\mathbf 1_{ \{ 1 \le \lvert \ell \rvert \le N \} } \mathbf 1_{ \{ 1 \le \lvert k_1 \rvert \le N/2 \} }\mathbf{1}_{\{|\ell\cdot\mathfrak w|\leq|\mathfrak w|\}}}
            {\lambda + \Gamma +(\mathfrak w \cdot \ell)^2\mathrm L^N_{k}\left(\lambda + \Gamma,z\right)}
            \lesssim \sum_{\ell\neq 0}\frac{\mathbf{1}_{\{|\ell\cdot\mathfrak w|<|\mathfrak w|\}}}{|\ell|^2}
            \lesssim 1\,.
        \end{equ}
        For the integral note that $\bigcup_{|\ell\cdot \mathfrak w|\leq \mathfrak w}Q_\ell^N$ is contained in $\{x\in\mathbb R^2: |x\cdot \mathfrak w|\leq \frac2N|\mathfrak w|\}$.
        Using this we see
        \begin{equ}
            \sum_{\substack{1\leq|\ell|\leq N\\|\mathfrak w\cdot\ell|\leq |\mathfrak w|}}
            \int_{Q^N_\ell}I_N(x)\dd x\leq \int_{|\mathfrak w\cdot x|\leq \frac{2}{N}|\mathfrak w|}I_N(x)\dd x\leq \int_{-2}^{2}\int_{-\infty}^\infty \frac{1}{\alpha+|x|^2}\dd x_1\dd x_2\lesssim 1\,,
        \end{equ}
        where we used a change of variables in $x$ (i.e. scaling by $N$) and the fact that $\alpha\geq 1$ as well as Lemma \ref{le:arctan_parameter}.
        Also note that the $x$ appearing in the rewriting of $\mathrm S_4$ (\ref{eq:S4Sum}) but not in $\mathrm S_5$, are contained in $\{x:1\leq|x|\leq 1+\frac{\sqrt{2}}{2N}\}$ and
        \begin{equ}
		\int_{1\le \lvert x \rvert \le 1+\tfrac{1}{2N}} I_N(x)\dd x \lesssim \sup_{1\le \lvert x \rvert \le 1+\tfrac{1}{2N}} \lvert I_N(x)\rvert \lesssim 1\,.
	\end{equ}
        It thus remains to show that 
	\begin{equ}\label{eq:last_one?}
		\sum_{\substack{1\le\lvert \ell \rvert \le N\\|\ell\cdot\mathfrak w|\geq|\mathfrak w|}} \int_{Q^N_{\ell}} \left\lvert I_N\left(\tfrac{\ell}{N}\right) - I_N(x) \right\rvert \dd x \lesssim 1 \, .
	\end{equ}
	In order to prove~(\ref{eq:last_one?}), we estimate, by the mean value theorem applied to the function $I_N$ and the line segment $[\tfrac{\ell}{N} , x]$, 
	\begin{equ} \label{eq:gradient_estimate}
		\left\lvert I_N\left(\tfrac{\ell}{N}\right) - I_N(x) \right\rvert
        \le
        \sup_{y \in Q^N_{\ell}} \left\lvert \nabla I_N(y) \right\rvert \left\lvert \tfrac{\ell}{N} - x \right\rvert
        \lesssim
        \frac{1}{N}(E+F+G)\, ,
	\end{equ}
	where $E$, $F$ and $G$ are the suprema over $Q_\ell^N$ of the the norms of the three terms in the expression of the gradient below:
	\begin{equ} \label{eq:grad_manipulated}
		 -\frac{1}{2}\nabla I_N(x) 
		 =\frac{x + (\mathfrak w \cdot x)\mathrm L^N_k(N^2(\lvert x\rvert^2 + \alpha_N),z)\mathfrak w+ \frac{(-\theta_k)(\mathfrak w \cdot x)^2  (\mathrm L^N(N^2(\lvert x\rvert^2 + \alpha_N),z))^{1-\theta_k}x}{(\lvert x \rvert^2 +\alpha_N)(\lvert x \rvert^2 +\alpha_N +1)}}
        {\left(\alpha_N+\lvert x \rvert^2 + (\mathfrak w \cdot x)^2\mathrm L^N_k(N^2(\lvert x\rvert^2 + \alpha_N),z)\right)^2} \, .
	\end{equ}
    Since each $Q_\ell^N$ has an area of $\frac{1}{N^2}$, we need to show that
	\begin{equ}\label{eq:last_one!}
		\sum_{\substack{1\le\lvert \ell \rvert \le N\\|\ell\cdot\mathfrak w|\geq|\mathfrak w|}} E+F+G \lesssim N^3 \, .
	\end{equ}
    Note first that for all $\ell$ such that $|\ell\cdot\mathfrak w|\geq |\mathfrak w|$ and $|\ell|\geq 1$ and for all $x\in Q_\ell^N$ it holds that
    \begin{equs}\label{eq:geometry}
        \left(1-\frac{\sqrt{2}}{2}\right)\left\lvert\frac{\ell}{N}\right\rvert &\leq \lvert x\rvert\leq \left(1+\frac{\sqrt{2}}{2}\right)\left\rvert\frac{\ell}{N}\right\rvert \, ,\\
        \left(1-\frac{\sqrt{2}}{2}\right)\left\rvert\mathfrak w\cdot\frac{\ell}{N}\right\rvert &\leq \left\rvert\mathfrak w\cdot x\right\rvert \leq \left(1+\frac{\sqrt{2}}{2}\right)\left\rvert\mathfrak w\cdot\frac{\ell}{N}\right\rvert\,.
    \end{equs}
    For $E$ note that
    \begin{equ}
        \sup_{x\in Q_\ell^N}\frac{|x|}{\left(\alpha_N+\lvert x \rvert^2 + (\mathfrak w \cdot x)^2\mathrm L^N_k(N^2(\lvert x\rvert^2 + \alpha_N),z)\right)^2}\lesssim \frac{N^3}{|\ell|^3}\,.
    \end{equ}
    For $F$ note that
    \begin{equ}
    \sup_{x\in Q_\ell^N}\frac{|\mathfrak w\cdot x|\mathrm L^N_k(N^2(\lvert x\rvert^2 + \alpha_N),z)|\mathfrak w|}{\left(\alpha_N+\lvert x \rvert^2 + (\mathfrak w \cdot x)^2\mathrm L^N_k(N^2(\lvert x\rvert^2 + \alpha_N),z)\right)^2}\lesssim \frac{N^3|\mathfrak w|}{|\ell|^2|\mathfrak w\cdot \ell|}\,,
    \end{equ}
    where we used one of the factors of the denominator to cancel the $\mathrm L_k^N$ in the numerator.
    Now note again by (\ref{eq:geometry}) that
    \begin{equ}
        \sum_{\substack{1\le\lvert \ell \rvert \le N\\|\ell\cdot\mathfrak w|\geq|\mathfrak w|}}\frac{\lvert\mathfrak w\rvert}{|\ell|^2|\mathfrak w\cdot\ell|}\lesssim \int_{|x\cdot\mathfrak w|\geq (1-\frac{\sqrt{2}}{2})\lvert\mathfrak w\rvert} \frac{|\mathfrak w|}{|x|^2|\mathfrak w\cdot x|}\dd x\lesssim 1\,.
    \end{equ}
    Finally for $G$ by similar arguments $G\lesssim \frac{N^3}{\ell^3}$, (note that $0\leq\theta_k\leq 1$).

	This completes step 5, with $C_5$ equal to the sum of the three constants with which we have estimated the sum corresponding to the $E$, $F$ and $G$ terms.

    \medskip
    \noindent {\bf Step 6} 
    Let $\theta_{\mathfrak w}$ be the angle from the first coordinate axis to $\mathfrak w$. By successively performing the change of variables $x \mapsto r (\cos \theta, \sin \theta)$ and $r^2 \mapsto r$, we first rewrite $\mathrm S_5$ as follows.
	\begin{equs}
		\mathrm S_5 & =\int_{0}^{2 \pi} \int_0^1 \frac{r}{r^2+\alpha_N+r^2 \lvert\mathfrak w\rvert^2 \cos^2(\theta-\theta_{\mathfrak w}) \mathrm L^N_k\left(N^2(r^2+\alpha_N),z\right)} \dd r \dd \theta \\
		& = \int_0^{2 \pi} \frac{1}{2}\int_0^1\frac{1}{r+\alpha_N +r \lvert\mathfrak w\rvert^2 \cos^2(\theta-\theta_{\mathfrak w}) \mathrm L^N_k\left(N^2(r+\alpha_N),z\right)} \dd r \dd \theta \\
		&= \int_0^{\pi} \int_0^1 \frac{1}{r+\alpha_N +r \lvert\mathfrak w\rvert^2 (\cos\theta)^2 \mathrm L^N_k\left(N^2(r+\alpha_N),z\right)} \dd r \dd \theta \, , \label{eq:int1}
	\end{equs}
	where in the last equality we used the $\pi$-periodicity of the integrand as a function of $\theta$.
    We then recall the definition of $\mathrm{S_6} \eqdef \mathrm{I}$ given in~(\ref{eq:moving_parts}).
    We observe that the absolute value of the difference between (\ref{eq:int1}) and (\ref{eq:moving_parts}), after the simplifications that occur, is upper bounded by	\begin{equ}\label{eq:int3}
		\int_0^{\pi} \int_0^1 \frac{(r+\alpha_N)^2(1+\lvert \mathfrak w\rvert^2(\cos\theta)^2\mathrm L^N_k)+ \alpha_N\lvert\mathfrak w \rvert^2 (\cos\theta)^2 \mathrm L^N_k}{\left[r+\alpha_N +r \lvert\mathfrak w\rvert^2 (\cos\theta)^2 \mathrm L^N_k\right]\left[(r+\alpha_N)(r+\alpha_N+1)(1+ \lvert\mathfrak w\rvert^2 (\cos\theta)^2 \mathrm L^N_k)\right]} \dd r \dd \theta \, ,
	\end{equ}
	where we omitted the argument of $\mathrm L^N_k$ for ease of reading. We thus study the two terms corresponding to the two summand of the numerator separately. The first one can be estimated by
\begin{equ}
		\int_0^{\pi} \int_0^1 \frac{(r+\alpha_N)}{\left[r+\alpha_N \right]\left[(r+\alpha_N+1)\right]} \dd r \dd \theta \le
		\int_0^{\pi} \int_0^1 \frac{1}{r+\alpha_N+1} \dd r \dd \theta \lesssim 1
	\end{equ}
	and the second one by
	\begin{equs}
		&\int_0^{\pi} \int_0^1 \frac{\alpha_N\lvert\mathfrak w \rvert^2 (\cos\theta)^2 \mathrm L^N_k}{\left[r+\alpha_N \right]\left[(r+\alpha_N)(r+\alpha_N+1)\lvert\mathfrak w\rvert^2 (\cos\theta)^2 \mathrm L^N_k\right]} \dd r \dd \theta \\
		&= \int_0^{\pi} \int_0^1 \frac{\alpha_N}{(r+\alpha_N)^2(r+\alpha_N+1)} \dd r \dd \theta \\
		& \lesssim \alpha_N \int_0^1 \frac{1}{(r+\alpha_N)^2} \dd r = \alpha_N \left(\frac{1}{\alpha_N}-\frac{1}{1+\alpha_N}\right) \le 1
		\, .    \end{equs}
    This concludes step 6 and with it also the proof of Lemma~\ref{le:replacement}.
\end{proof}

\begin{lemma}[From integral to estimate] \label{le:int_est}
	Recall the definition of $\tilde{\mathrm S}$ given in~(\ref{eq:moving_parts_S}) (in particular that it depends on $k$). There exists a constant $C_{\mathrm{Diag}}$ such that, for any even $k\ge 2$,
	\begin{equ}
		\tilde{\mathrm S} \le \frac{3\pi}{2\lvert \mathfrak w\rvert}\left(1+\frac{\lvert \mathfrak w \rvert C_{\mathrm{Diag}}}{z^{\theta_{k+1}}}\right) \mathrm L^N_{k+1}\left(\lambda +\lvert k_{1:n}\rvert^2,z\right) , \label{eq:S_ub}
    \end{equ}
    whereas, for any odd $k\ge 3$,
    \begin{equ}
		\tilde{\mathrm S} \ge  \frac{3\pi}{2\lvert \mathfrak w\rvert} \left[\left(1-\left(\lvert \mathfrak w\rvert C_{\mathrm{Diag}}+2+\frac{3}{\lvert \mathfrak w \rvert}\right)\frac{1}{z^{\frac{\theta_k}{2}}}\right) \mathrm L^N_{k+1}(\alpha,z) - \frac{4}{3}z^{\theta_{k+1}} \right] , \label{eq:S_lb}
	\end{equ}
	uniformly in $\lambda \in (0,+\infty)$, $z \in (1,+\infty)$, $k_{1:n}\in  (\mathbb Z^2_0)^n$, $N\in \mathbb N$ and $\mathfrak w\in\mathbb R^2$.
\end{lemma}
\begin{proof}
	By applying Lemma~\ref{le:replacement}, we immediately get
    \begin{equ} \label{eq:double_c}
        \lvert \tilde{\mathrm S} - \mathrm{I}\rvert \le C_{\mathrm{Diag}} \, .
    \end{equ} 
    The task now is to obtain bounds on $\mathrm I$, whose definition was given in~(\ref{eq:moving_parts}).
    
    We start by proving the statement for $k\ge4$, which corresponds to $\theta_k \in \left[\tfrac12,\tfrac34\right]$. We first transform integral $\mathrm I$ with the change of variables
	\begin{equ} \label{eq:change_var}
		u = \mathrm L^N\left(N^2(r+\alpha_N),z\right) = \log\left(1+\frac{1}{r+\alpha_N}\right)+z \, , \quad \dd u = \frac{-1}{(r+\alpha_N)(r+\alpha_N+1)}\dd r \, , 
	\end{equ}
	which gives
	\begin{equ} \label{eq:I_changed}
		\mathrm I = \int_0^{\pi} \int_{\mathrm L^N(N^2(1+\alpha_N),z)}^{\mathrm L^N(\alpha,z)} \frac{1}{1+\lvert\mathfrak w\rvert^2 (\cos\theta)^2 u^{\theta_k}} \dd \theta \dd u \, . 
	\end{equ}
	By then integrating in $\theta$, using Lemma~\ref{le:arctan_parameter} with $\gamma=\lvert\mathfrak w\rvert^2 u^{\theta_k}$ and $\beta=1$, we get
	\begin{equ}
		\int_{\mathrm L^N(N^2(1+\alpha_N),z)}^{\mathrm L^N(\alpha,z)} \frac{\pi}{\sqrt{1+\lvert\mathfrak w\rvert^2 u^{\theta_k}}} \dd u =\frac{\pi}{\lvert \mathfrak w \rvert} (\mathrm I_1-\mathrm I_2) \, , 
	\end{equ}
	where $\mathrm I_1,\mathrm I_2 >0$ are defined by
	\begin{equ}
		\mathrm I_1 \eqdef  \int_{\mathrm L^N(N^2(1+\alpha_N),z)}^{\mathrm L^N(\alpha,z)} \frac{1}{\sqrt{u^{\theta_k}}} \dd u \, , \qquad \mathrm I_2 \eqdef \int_{\mathrm L^N(N^2(1+\alpha_N),z)}^{\mathrm L^N(\alpha,z)} \left(-\frac{1}{\sqrt{\lvert\mathfrak w\rvert^{-2} +u^{\theta_k}}} + \frac{1}{\sqrt{u^{\theta_k}}}\right) \dd u \, .
	\end{equ}
	The integral $\mathrm I_1$ is the one announced in the general strategy explained before Lemma~\ref{le:replacement}. Indeed, it can be computed explicitly:
	\begin{equ} \label{eq:int_I1}
		\mathrm I_1 = \frac{\left(\mathrm L^N(\alpha,z)\right)^{\theta_{k+1}}}{\theta_{k+1}} - \frac{\left(\mathrm L^N(N^2(1+\alpha_N),z)\right)^{\theta_{k+1}}}{\theta_{k+1}} = \frac{\mathrm L^N_{k+1}(\alpha,z)}{\theta_{k+1}} - \frac{\mathrm L^N_{k+1}(N^2(1+\alpha_N),z)}{\theta_{k+1}} \, .
	\end{equ}
	It is precisely this computation that gives upper and lower bounds of the form logarithm to the power $\theta_k$, with the sequence of powers $(\theta_k)_{k\ge 2}$ converging to $2/3$.
	The integral $\mathrm I_2$, instead, is regarded as an error term and can be estimated by
	\begin{equs}
		\mathrm I_2 & = \int_{\mathrm L^N(N^2(1+\alpha_N),z)}^{\mathrm L^N(\alpha,z)} \frac{\sqrt{\lvert\mathfrak w\rvert^{-2} +u^{\theta_k}}-\sqrt{u^{\theta_k}}}{\sqrt{u^{\theta_k}} \sqrt{\lvert\mathfrak w\rvert^{-2} +u^{\theta_k}}} \dd u  \\
		&\le\int_{\mathrm L^N(N^2(1+\alpha_N),z)}^{\mathrm L^N(\alpha,z)} \frac{\sqrt{u^{\theta_k}} + \sqrt{\lvert\mathfrak w\rvert^{-2}} - \sqrt{u^{\theta_k}}}{\sqrt{u^{\theta_k}} \sqrt{u^{\theta_k}}} \dd u  \\
		&=\lvert\mathfrak w \rvert^{-1} \int_{\mathrm L^N(N^2(1+\alpha_N),z)}^{\mathrm L^N(\alpha,z)} \frac{1}{u^{\theta_k}} \dd u  \\
		&=\frac{1}{\lvert\mathfrak w \rvert(1-\theta_k)} \left[\left(\mathrm L^N(\alpha,z)\right)^{1-\theta_{k}} - \left(\mathrm L^N(N^2(1+\alpha_N),z)\right)^{1-\theta_{k}}  \right] \, . \label{eq:ll_I2}
	\end{equs}
	The inequality
	\begin{equ}
		\theta_{k+1} - (1-\theta_k) = \frac{\theta_k}{2} > 0
	\end{equ}
	shows that $\mathrm I_2(N) = o_{N \to \infty}(\mathrm I_1(N))$, so that $\mathrm I_2$ is indeed of lower-order in $N$ with respect to $\mathrm I_1$.

	We now have everything we need to conclude the proof in the case $k\ge4$.
	We first show the upper bound~(\ref{eq:S_ub}). Inequality~(\ref{eq:double_c}) and the above steps give
	\begin{equ} \label{eq:case_k_2}
		\tilde{\mathrm{S}} \le \frac{\pi}{\lvert \mathfrak w\rvert}(\mathrm I_1 -\mathrm I_2) + C_{\mathrm{Diag}}  \, .
	\end{equ}
	Then, by plugging~(\ref{eq:int_I1}) and~(\ref{eq:ll_I2}) into~(\ref{eq:case_k_2}) above and dropping the negative terms, we obtain
	\begin{equs} 
		\tilde{\mathrm S} &
        \le
        \frac{\pi}{\lvert \mathfrak w\rvert\theta_{k+1}} \left(\mathrm L^N(\alpha,z)\right)^{\theta_{k+1}} + C_{\mathrm{Diag}} \\
		& \le
        \frac{\pi}{\lvert \mathfrak w\rvert\theta_{k+1}}\left(\mathrm L^N(\alpha,z)\right)^{\theta_{k+1}} + C_{\mathrm{Diag}}\frac{\left(\mathrm L^N(\alpha,z)\right)^{\theta_{k+1}}}{z^{\theta_{k+1}}} \\
        & \le
        \frac{3\pi}{2\lvert \mathfrak w\rvert} \left(1+\frac{\lvert \mathfrak w \rvert C_{\mathrm{Diag}}}{z^{\theta_{k+1}}}\right) \mathrm L^N_{k+1}(\alpha,z) \, , \qquad \quad  \label{eq:usefulness_of_z_1}
	\end{equs}
	where we observed that, for $k$ even, $\theta_{k+1}\in\left[\tfrac23,1\right]$. What we obtained is exactly the claimed~(\ref{eq:S_ub}).

	Finally, we prove the lower bound~(\ref{eq:S_lb}). Inequality~(\ref{eq:double_c}) and the above steps give
	\begin{equ}
		\tilde{\mathrm S} \ge \frac{\pi}{\lvert \mathfrak w\rvert}(\mathrm I_1 - \mathrm I_2) - C_{\mathrm{Diag}} \, .
	\end{equ}
	By dropping the second term in the square brackets in~(\ref{eq:ll_I2}) and using $1+\alpha_N \le 2$ to estimate the negative term in~(\ref{eq:int_I1}) by $\mathrm L_{k+1}(1+\alpha_N,z) \le (\log2+z)^{\theta_{k+1}} \le (\log2)^{\theta_{k+1}}+z^{\theta_{k+1}}$, we obtain 
	\begin{equs}
		\tilde{\mathrm S}
        &\ge
        \frac{\pi}{\lvert \mathfrak w\rvert}\left[\frac{\left(\mathrm L^N(\alpha,z)\right)^{\theta_{k+1}}}{\theta_{k+1}} - \frac{(\log2)^{\theta_{k+1}}}{\theta_{k+1}} -\frac{z^{\theta_{k+1}}}{\theta_{k+1}} - \frac{1}{\lvert\mathfrak w \rvert(1-\theta_k)} \left(\mathrm L^N(\alpha,z)\right)^{1-\theta_{k}}\right] -  C_{\mathrm{Diag}} \\
		&\ge
        \frac{3\pi}{2\lvert \mathfrak w\rvert} \left[\left(1-\frac{\lvert \mathfrak w \rvert C_{\mathrm{Diag}}+2}{z^{\theta_{k+1}}} - \frac{3}{\lvert\mathfrak w \rvert z^{\frac{\theta_k}{2}}}\right) \mathrm L^N_{k+1}(\alpha,z) - \frac{4}{3}z^{\theta_{k+1}} \right] \\
        &\ge 
        \frac{3\pi}{2\lvert \mathfrak w\rvert} \left[\left(1-\left(\lvert \mathfrak w\rvert C_{\mathrm{Diag}}+2+\frac{3}{\lvert \mathfrak w \rvert}\right)\frac{1}{z^{\frac{\theta_k}{2}}}\right) \mathrm L^N_{k+1}(\alpha,z) - \frac{4}{3}z^{\theta_{k+1}} \right] ,
        \label{eq:usefulness_of_z_2}  \\
	\end{equs}
	where for the first inequality we observed that, for $k$ odd, $\theta_{k+1}\in \left[\tfrac12,\tfrac23\right]$ and then applied similar steps as the ones used to obtain (\ref{eq:usefulness_of_z_1}), whereas in the second one we used that $\theta_{k+1}\ge\frac{\theta_{k}}{2}$. What we obtained is exactly the claimed~(\ref{eq:S_lb}).

    If $k=3$, we proceed as above until we reach we reach~(\ref{eq:I_changed}). At this point, we do not need to split $\mathrm I$ into $\mathrm I_1$ and $\mathrm I_2$, since in this case $\mathrm I$ already admits an explicit primitive. Indeed, by first using Lemma~\ref{le:arctan_parameter} to integrate in $\theta$ and then applying the change of variables $v = 1 + \lvert \mathfrak w \rvert^2 u$, we obtain 
    \begin{equ}
        \mathrm I = \left. \frac{2\pi}{\lvert \mathfrak w \rvert^2} \sqrt{1+\lvert \mathfrak w\rvert^2 u} \right|_{\mathrm L^N(N^2(1+\alpha_N),z)}^{\mathrm L^N(\alpha,z)} \, ,
    \end{equ}
    so that
    \begin{equs}
        \tilde{\mathrm S}
        & \ge
        \mathrm I -  C_{\mathrm{Diag}} \\
        & \ge 
        \frac{3 \pi}{2 \lvert \mathfrak w \rvert} \left[\left(1-\frac{\lvert \mathfrak w \rvert C_{\mathrm{Diag}}}{z^{\theta_4}}\right)\mathrm L_4^N(\alpha,z) - \frac{4}{3}\left(\frac{1}{\lvert \mathfrak w \rvert^2}+\log(2)+z\right)^{\theta_4} \right] \\
        & \ge 
        \frac{3 \pi}{2 \lvert \mathfrak w \rvert} \left[\left(1-\left(\lvert \mathfrak w \rvert C_{\mathrm{Diag}}+2+\frac{3}{\lvert \mathfrak w \rvert}\right)\frac{1}{z^{\theta_4}}\right)\mathrm L_4^N(\alpha,z) - \frac{4}{3}z^{\theta_4} \right] , \\
    \end{equs}
    where we recalled that $\theta_4 = \tfrac12$ and followed steps similar to the ones with which we obtained (\ref{eq:usefulness_of_z_2}). The claimed \ref{eq:S_lb} immediately follows by the trivial $\theta_4\ge\frac{\theta_4}{2}$.
    
    If $k=2$, $\mathrm L^N_k \equiv 1$, because $\theta_2 = 0$ by definition. In this somewhat degenerate case, we estimate from above $\mathrm I$, defined in~\ref{eq:moving_parts}, using the following steps. First apply Lemma~\ref{le:arctan_parameter} to integrate it in $\theta$. Then lower bound the resulting factor $(r+\alpha_N+1)$ in the denominator by $1$ and integrate in $r$. The upper bound obtained in this way is
    \begin{equ}
        \frac{\pi}{\lvert \mathfrak w \rvert} \log\left(1+\frac{1}{\alpha_N}\right) \le \frac{\pi}{\lvert \mathfrak w \rvert} \mathrm L^N_3\left(\lambda + \lvert k_{1:n} \rvert^2,z\right) \, . \label{eq:easy_case}
    \end{equ}
    Finally, the stated~(\ref{eq:S_ub}) follows from~(\ref{eq:double_c}) and steps analogous to the ones that concluded the case $k\ge4$.
\end{proof}

\begin{remark}\label{re:to_add_or_to_multiply?}
    In inequalities~(\ref{eq:usefulness_of_z_1}) and~(\ref{eq:usefulness_of_z_2}), we first go from an additive error to a multiplicative one, which is easier to iterate, but increases complexity, and then decrease complexity by estimating $\mathrm{L}^N(\alpha,z)$ by $z$.
    This is quite rough, but otherwise the iteration would give more and more complicated bounds at each step.
\end{remark}

\section{Heuristic Derivation of the Green Kubo Formula}\label{ap:bulkheuristic}
In this Appendix we give a heuristic derivation of the bulk diffusivity formula (\ref{eq:bulkD}).

Consider the equation on the full space regularized with Fourier cut-off 1 with regularized white noise:
\begin{equ}[eq:smoothNoise]
	\partial_t\eta =\frac{1}{2}\Delta\eta+ \mathfrak w\cdot\Pi_1\nabla(\Pi_1\eta)^2  + \nabla\cdot\Pi_a\xi\,,
\end{equ}
where $\Pi_a \xi$ is a space time white noise, regularized in space by a cut-off in Fourier at level $a\in(1,\infty)$.
It can be seen using techniques adopted in \cite{CESnontriviality} that this equation still has a unique solution, existing for all time, and this solution is a strong Markov process invariant under translations in space and time.
The invariant measure is given by regularized spacial white noise $\eta^a$, regularized by the same cut-off as $\Pi_a\xi$.
We consider the equation run at stationarity, i.e. started from $\eta^a$.
Since the noise is regularized, $\eta$ will be a continuous function and therefore we can evaluate it at space-time points.
This allows us to define the correlation function
\begin{equ}
	S(t,x)=\mathbf{E}(\eta(t,x)\eta(0,0))\,,
\end{equ}
for $t\geq 0$ and $x\in\mathbb R^2$.

The bulk diffusivity is commonly defined (see e.g. \cite{Yau2004logtProcess,spohn2012large} for examples from discrete systems and \cite{BQS2011KPZfluctuations} for a continuous example in $d=1$) as the matrix $(D_{ij}(t))_{1\leq i,j\leq 2}$ with entries given by
\begin{equ}
	D_{ij}(t)=\frac{1}{2t}\int_{\mathbb{R}^2}x_ix_jS(t,x)\dd x \, .
\end{equ}
Without loss of generality assume that $\mathfrak w_2=0$, then the reflection symmetry of the system in the second component gives $D_{12}(t)=D_{21}(t)=0$.
We will work with the bulk diffusivity as defined in \cite{CET2023stationary}
\begin{equ}
	D(t)=\frac{1}{2t}\int_{\mathbb{R}^2}\lvert x\rvert ^2S(t,x)\dd x\,,
\end{equ}
which can be interpreted as ($1/t$ times) the variance of $S(t,\cdot)$ seen as a density.
In a particle system this would be the density of a second class particle started at the origin.
This definition of the bulk diffusivity can be connected to the bulk diffusivity matrix above by taking the trace, see also the remark at the end of the section.

We now want to show that this definition of the bulk diffusivity is heuristically consistent with (\ref{eq:bulkD}).
To do this assume that $S(t,x)$ decays fast in $\lvert x\rvert $, noting that for the linear case (i.e. $\mathfrak w=0$) it is the density of a multivariate Gaussian.
Also assume that $S(t,\cdot)$ integrates to $1$ for every $t$.
At time $t=0$ this is true by the law of the stationary measure, since cutting Fourier-modes larger than $1$ is equivalent to convolving with a mass $1$ bump function.
For later time it formally follows from the conservative nature of (\ref{eq:smoothNoise}):
\begin{equs}
	\int_{\mathbb{R}^2} S(t,x)\dd x &=\int_{\mathbb{R}^2}\mathbf E(\eta(t,x)\eta(0,0))\dd x=\mathbf E\left(\eta(0,0)\int_{\mathbb{R}^2} \eta(t,x)\dd x\right)\\
	&=\mathbf E\left(\eta(0,0)\int_{\mathbb{R}^2} \eta(0,x)\dd x\right)=\int_{\mathbb{R}^2} S(0,x)\dd x=1 \,,
\end{equs}
where the third equality follows from an integration by parts, because the entire right-hand side of (\ref{eq:smoothNoise}) can be put in divergence form.
Integrating (\ref{eq:smoothNoise}) in time, multiplying by $\eta(0,0)$ and taking expectations we obtain:
\begin{equ}\label{eq:Sintegrals}
	S(t,x)=S(0,x)+\frac12\int_0^t\Delta S(s,x)\dd s+ \int_0^t\mathbf E\left(\mathcal{N}(\eta)(s,x)\eta(0,0)\right)\dd s\, ,
\end{equ}
where the noise term disappears because it is centered and independent of $\eta(0,0)$, and $\mathcal{N}$ is
\begin{equ}
	\mathcal N(\eta)=\mathfrak w\cdot\Pi_1\nabla\left(\Pi_1\eta\right)^2\,.
\end{equ}
We will integrate the terms on the right-hand side against $\lvert x\rvert ^2$ and divide them by $2t$ one by one.
The first one does not depend on time before dividing by $2t$ and so will vanish for large $t$.
The second one is
\begin{equ}
	\frac{1}{4t}\int_0^t\int_{\mathbb R^2}\lvert x\rvert ^2\Delta S(t,x)\dd x=\frac{1}{t}\int_0^t\int_{\mathbb R^2}S(t,x)\dd x=1 \, .
\end{equ}
Finally let us consider the third one.
Using that $\mathcal N$ is quadratic in $\eta$ and $\eta$ is Gaussian we see that
\begin{equ}\label{eq:cubezero}
	\mathbf E\left(\mathcal N(\eta)(s,x)\eta(s,0)\right)=0 
\end{equ}
and so we can we rewrite
\begin{equs}
	\int_{\mathbb R^2}\lvert x\rvert ^2\mathbf E\left(\mathcal{N}(\eta)(s,x)\eta(0,0)\right)\dd x
	=&\int_{\mathbb R^2}\lvert x\rvert ^2\mathbf E\left(\mathcal{N}(\eta)(s,x)(\eta(0,0)-\eta(s,0))\right)\dd x\\
	=&\int_{\mathbb R^2}\lvert x\rvert ^2\mathbb E\left(\mathcal{N}(\xi^a)(x)\tilde{\mathbf E}_{\xi^a}(\tilde{\eta}(s,0)-\tilde{\eta}(0,0))\right) \dd x\,,
\end{equs}
where $\tilde{\mathbf E}_{\xi^a}$ is the expectation with respect to the law of $\tilde{\eta}(r,x)\eqdef \eta(s-r,x)$, conditioned on $\xi_a\eqdef\tilde{\eta}(0,\cdot)$ and $\mathbb E$ is the expectation associated to the law of a (mollified) spacial white noise $\xi_a$.
Furthermore we used translation invariance of $\xi^a$.
The time reversed process $\tilde{\eta}$ satisfies the equation (\ref{eq:smoothNoise}) with a changed sign in front of the nonlinearity and a different noise with the same law, that is independent of $\xi^a$.
Using this we get
\begin{equs}
	\mathbb E\left(\mathcal{N}(\xi^a)(x)\tilde{\mathbf E}_{\xi^a}(\tilde{\eta}(s,0)-\tilde{\eta}(0,0))\right)
	=&\int_0^s\mathbb E\left(\mathcal{N}(\xi^a)(x)\tilde{\mathbf E}_{\xi^a}(\Delta\tilde{\eta}(r,0)- \mathcal N(\tilde{\eta})(r,0))\right)\dd r
    \\=&
    \int_0^s\mathbb E\left(\mathcal{N}(\xi^a)(x)\tilde{\mathbf E}_{\xi^a}(\Delta\tilde{\eta}(r,0))\right)\dd r\label{eq:laplaceterm}\\
    &+\int_0^s\mathbb E\left(\mathcal{N}(\xi^a)(x)\tilde{\mathbf E}_{\xi^a}(- \mathcal N(\tilde{\eta})(r,0))\right)\dd r\label{eq:Nterm}\,,
\end{equs}
where we used that the noise term vanishes under the expectation, since the noise is independent of $\xi^a$.
This is the reason for considering the time reversed process.
The term \eqref{eq:laplaceterm}, integrated against $\lvert x\rvert ^2$, using the translation invariance to move the $x$ to the $\tilde{\eta}$ and an integration by parts, becomes
\begin{equs}
	4\int_0^s\int_{\mathbb R^2}\mathbb E\left(\mathcal{N}(\xi^a)(0)\tilde{\mathbf E}_{\xi^a}( \tilde{\eta}(r,x)\right) \dd x \dd r
	& =4\int_0^s\int_{\mathbb R^2}\mathbf E\left(\mathcal{N}(\xi^a)(0)\eta(s-r,-x)\right)\dd x \dd r\\
	& =4\int_0^s\int_{\mathbb R^2}\mathbb E\left(\mathcal{N}(\xi^a)(0)\eta(0,-x)\right)\dd x \dd r=0\,,
\end{equs}
where we used that the dynamics are conservative and then again (\ref{eq:cubezero}).

The term \eqref{eq:Nterm} integrated against $\lvert x\rvert ^2$ becomes
\begin{equs}
	-\int_{\mathbb R^2} \lvert &x\rvert ^2\int_0^s \mathbb E\left(\mathcal{N}(\xi^a)(x)\tilde{\mathbf E}_{\xi^a}(\mathcal N(\tilde{\eta})(r,0))\right)\dd r \dd x\\
	&=-\int_0^s\int_{\mathbb R^2} \lvert x\rvert ^2\mathbf E\left(\mathcal{N}(\eta)(s,x)\mathcal N(\eta)(s-r,0)\right)\dd x \dd r\\
	&=-\int_0^s\int_{\mathbb R^2}\lvert x\rvert ^2\mathbf E\left(\mathcal{N}(\eta)(r,x)\mathcal N(\eta)(0,0)\right)\dd x \dd r\\
	&=-\int_0^s\int_{\mathbb R^2}\lvert x\rvert ^2\mathbf E\left(((\mathfrak w\cdot\nabla)\Pi_1{:}(\Pi_1\eta)^2{:})(r,x)((\mathfrak w\cdot\nabla)\Pi_1{:}(\Pi_1\eta)^2{:})(0,0)\right)\dd x \dd r\\
	&=2\lvert\mathfrak w\rvert^2\int_0^s\int_{\mathbb R^2}\mathbf{E}\left((\Pi_1{:}(\Pi_1\eta)^2{:})(r,x)(\Pi_1{:}(\Pi_1\eta)^2{:})(0,0)\right)\dd x \dd r\,,
\end{equs}
where in the last step we first performed integration by parts on the gradient from the first factor, and then used translation invariance to move the $x$ to the second factor, after which we perform another integration by parts.
Each integration by parts gives a factor $-1$, as well as an additional $-1$, since the $x$ becomes a $-x$ when moved to the second factor.
These integration by parts are not rigorous, since we cannot exchange the integral and the expectation.
However the terms in each line are well-defined assuming the decay in $S$ mentioned above.
Here the Wick squares ${:}X^2{:}$ simply subtract the expectations, i.e. ${:}X^2{:}=X^2-\mathbf{E}(X^2)$.
They are necessary since otherwise the integrand would not decay in space.
Collecting all the terms in equation~\ref{eq:Sintegrals} we obtain:
\begin{equ}
	D(t)=1+\frac{\lvert\mathfrak w\rvert^2}{t}\int_0^t\int_0^s\int_{\mathbb R^2} \mathbf{E}\left((\Pi_1{:}(\Pi_1\eta)^2{:})(0,0)(\Pi_1{:}(\Pi_1\eta)^2{:})(r,x)\right) \dd x\dd r \dd s+o(1)\,.
\end{equ}
Dropping the $o(1)$ term and replacing $\mathbb{R}^2$ with a large torus $\mathbb T_N^2$ we obtain exactly formula (\ref{eq:bulkD}).

If we had chosen instead to analyse $D_{ij}(t)$ the same steps would have given, for $\mathfrak w$ parallel to the first coordinate axis,
\begin{equ}
	D_{11}(t)=D(t)-1/2,\quad D_{12}(t)=D_{21}(t)=0,\quad D_{22}(t)=1/2\,.
\end{equ}
As we see, $D_{11}$ and $D$ have equivalent asymptotic behaviour.

\section*{Acknowledgements}
 The~authors would like to thank Fabio Toninelli for invaluable and continuous support and Giuseppe Cannizzaro and Quentin Moulard for useful discussions.
D. D. G. gratefully acknowledges financial support of the  Austrian Science Fund (FWF), Project Number P 35428-N.



\bibliographystyle{alpha}
\bibliography{references}

\end{document}